\documentclass{amsart}
\usepackage[utf8]{inputenc}
\usepackage{amsfonts,latexsym,amssymb,amsmath,amsthm,color}
\usepackage{enumerate}
\usepackage{captdef}
\usepackage[dvips]{graphicx} 
\usepackage{enumerate,graphicx,graphpap}
\usepackage[matrix]{xy}
\usepackage{cases}
\usepackage{mathrsfs}
\usepackage{enumitem}
\input xy
\usepackage{cite}

\theoremstyle{plain}
\newtheorem{teor}{Theorem}[section]
\newtheorem{lema}[teor]{Lemma}
\newtheorem{coro}[teor]{Corollary}
\newtheorem{prop}[teor]{Proposition}

\theoremstyle{definition}
\newtheorem{defi}[teor]{Definition}
\newtheorem{eje}{Example}[section]
\newtheorem{nota}[teor]{Remark}

\newtheoremstyle{teoremacita}
{3pt}
{3pt}
{\itshape}
{}
{\bfseries}
{}
{ }
{\thmname{#1}\thmnumber{ #2'}\thmnote{ #3}.}
\theoremstyle{teoremacita} \newtheorem*{teor*}{}

\def\N{\mathbb N}

\def\R{\mathbb R}
\def\C{\mathbb C}

\def\N{\mathbb N}

\def\R{\mathbb R}
\def\C{\mathbb C}

\def\a{\alpha}
\def\b{\beta}

\def\rr{{\boldsymbol r}}

\def\xx{{\boldsymbol x}}
\def\yy{{\boldsymbol y}}
\def\uu{{\boldsymbol u}}

\def\xxi{{\boldsymbol \xi}}
\def\aa{{\boldsymbol \alpha}}
\def\bb{{\boldsymbol \beta}}

\def\ee{{\boldsymbol \varepsilon}}
\def\00{{\boldsymbol 0}}
\def\11{{\boldsymbol 1}}
\def\zz{{\boldsymbol z}}

\def\PP{{P}}

\def\d{\partial}

\def\ll{\boldsymbol{\lambda}}

\begin{document}

\title{Formal $P$-Gevrey series solutions of first order holomorphic PDEs\label{ra_ch1}}

\author{Sergio A. Carrillo}
\address{Escuela de Ciencias Exactas e Ingenier\'{i}a, Universidad Sergio Arboleda, Calle 74, $\#$ 14-14, Bogot\'{a}, Colombia.}
\email{sergio.carrillo@usa.edu.co}

\author{Carlos A. Hurtado}
\address{Universidad Privada del Norte, Campus Bre\~{n}a, Avenida Tingo Maria 1122, Lima, Per\'{u}}
\email{carlos.hurtado@upn.edu.pe}




\begin{abstract} We provide a complete and self-contained proof of the Gevrey character, in an analytic function $P$, of formal power series solutions of some families of first order holomorphic PDEs.  Our approach is based on a majorant series technique by applying Nagumo norms joint with a division algorithm.
\end{abstract}

\maketitle

\vspace{-0.5cm}

\section{Introduction}

In the study of singular ordinary and partial differential equations or in the case of singular perturbation problems, a technique to obtain holomorphic solutions from formal ones is by applying certain summability methods such as Borel--summability and multisummability. These solutions represent asymptotically the formal power series solution as the variables approach the singular locus in adequate domains. In general, the first step to follow this method is to determine the existence, uniqueness and divergence rate (Gevrey order) of these series. The study of their summability is determined by the nature of the equation and it is a much harder problem. We refer to Refs. \cite{BalserLoday09,BalserMozo02,CDMS,Carr1,CostinTanveer2007,CostinTanveer2009,Hibino2006I,Hibino2006II,Klimes2016,LastraMalek2015,Ouchi02,TahataYamazawa2013,Zhang19,YamazawaYoshino15,Remy20} for some examples of ODEs and PDEs, which are susceptible to this type of analysis.

The goal of this paper is to provide a self-contained proof on the Gevrey type of formal power series solutions $\widehat{\yy}$ of holomorphic partial differential equations of first order at a singular locus $S$. We will show that under a suitable geometric condition, the germ of analytic function $P$ that generates $S$ is the generic source of divergence: $\widehat{\yy}$ is a $P$-$1$--Gevrey series. Roughly speaking, this means that we can write $\widehat{\yy}=\sum_{n=0}^\infty y_n P^n$ as a power series in $P$, where the coefficients $y_n$ are holomorphic in a common polydisc $D$ at the origin and $\sup_{\xx\in D}|y_n(\xx)|\leq CA^n n!$, for some constants $C,A>0$. More specifically, if $\xx=(x_1,\dots,x_d)\in(\C^d,\00)$ and $\yy=(y_1,\dots,y_N)\in\C^N$, 
we consider a germ $P$ of a non-zero holomorphic function on $(\C^d,\00)$ such that $P(\00)=0$, and the system of partial differential equations \begin{equation}\label{Eq. Main Eq}
	P(\xx)L(\yy)(\xx)=F(\xx,\yy),\quad \text{ where } L:=a_1(\xx)\d_{x_1}+\cdots+a_d(\xx)\d_{x_d},
\end{equation}  is a first order differential operator with holomorphic coefficients $a_j$ near the origin -not all identically zero-, and $F$ is a $\C^N$-valued holomorphic map defined near $(\00,\00)\in\C^d\times\C^N$. The \textit{singular locus} of (\ref{Eq. Main Eq}) is the germ at the origin of the analytic set \[S:=\{\xx\in(\C^d,\00) : P(\xx)a_j(\xx)=0, j=1,\dots,d\},\] where the nature of equation (\ref{Eq. Main Eq}) changes from differential to implicit one. Note that $S$ contains the zero set of $P$, and both coincide if $a_j(\00)\neq 0$, for at least some index $j$. Furthermore, if $\frac{\partial F}{\partial \yy}(\00,\00)$ is an invertible matrix, $P$ cannot be canceled from (\ref{Eq. Main Eq}), so its zero set is a  non-removable singular part of (\ref{Eq. Main Eq}). Under these conditions our main result can be stated as follows.

\begin{teor}\label{Thm Main Result} Consider the partial differential equation (\ref{Eq. Main Eq}) where $F(\00,\00)=\00$, and $\mu:=\frac{\partial F}{\partial \yy}(\00,\00)$ is an invertible matrix. If $P$ divides $L(P)$, equation (\ref{Eq. Main Eq}) has a  unique formal power series solution $\widehat{\yy}\in \C[[\xx]]^N$. Moreover, $\widehat{\yy}$ is a $P$-$1$--Gevrey series.
\end{teor}

The notion of $P$-Gevrey series was introduced by J. Mozo-Fern\'{a}ndez and R. Sch\"{a}fke in Ref. \cite{MS19} in the framework of asymptotic expansions and summability with respect to a germ of an analytic function, and it generalizes the notion of Gevrey series in one variable. Our aim is prove Theorem \ref{Thm Main Result} resorting only on the ideas of this recent theory instead of using previous results on divergent solutions of systems of holomorphic PDEs, see, e.g., \cite{GerardTahara}.

Equation (\ref{Eq. Main Eq}) falls into the category of singular first order PDEs  \begin{equation}\label{Eq. Main 2}
	L_1(\yy)(\xx)=F(\xx,\yy),
\end{equation} where $L_1=\sum_{j=1}^d X_j(\xx) \d_{x_j}$ is a germ of a  holomorphic vector field, singular at $\00\in\C^d$, i.e., $X_j(\00)=0$, for all $j=1,\dots,d$. The convergence vs. rate of divergence of formal power series solutions of (\ref{Eq. Main 2}) has been studied extensively by several authors, see, e.g., Refs.  \cite{Oshima73,Kaplan79,Hibino99,Hibino04,Yamazawa2000,Ouchi05}. These growth properties depend on conditions on $S=\{\xx\in\C^d : X_j(\xx)=0, j=1,\dots,d\}$ or on its associated ideal $(X_1,\dots,X_d)\subseteq \C\{\xx\}$, and on non-resonance conditions on $\mu$ and the Jacobian matrix $\Lambda:=\left(\d_{x_i}X_j(\00)\right)_{i,j}$ that we will explain below. Then, if $S$ is an analytic submanifold, c.f., Refs. \cite{Kaplan79,Yamazawa2000}, by choosing a suitable analytic coordinate system $\xxi=(\xi_1,\dots,\xi_d)$ of $(\C^d,\00)$ where $S$ is the zero set of some of these coordinates, and $\Lambda$ is in canonical Jordan form, the convergence or a Gevrey type of solutions can be obtained. For a recent account on these results see Ref. \cite{LastraTahara19} and the references therein.

Set $\text{Spec}(\mu)=\{\mu_1,\dots,\mu_d\}$ and  $\text{Spec}(\Lambda)=\{\lambda_1,\dots\lambda_m,0,\dots,0\}$, where $\mu_k\neq0$, $\lambda_j\neq 0$, and all eigenvalues are repeated according multiplicity. If $m\geq 1$, the classical non-resonance \textit{Poincar\'{e} condition} requests that   \begin{equation}\label{Poincare condition}
	|\lambda_1\b_1+\dots+\lambda_m\b_m-\mu_k|\geq \nu |\bb|,\quad \text{ for all } \bb\in\N^m, k=1,\dots,d,
\end{equation} for some constant $\nu>0$, see, e.g., \cite[p. 71]{Chow94}, \cite[p. 166]{Kaplan79}. Then, if (\ref{Poincare condition}) is valid and $m=d$, i.e., $\Lambda$ is invertible, the solution of (\ref{Eq. Main 2}) is convergent, see Refs. \cite{Oshima73,Hibino99,Hibino04}. Otherwise, the solution is generically divergent, but of some Gevrey order in the variable $\xxi$, depending on the sizes of the blocks of the canonical Jordan form of $\Lambda$ associated to the zero eigenvalue, see Refs.  \cite{Hibino99,Hibino04}. It is worth remarking that (\ref{Poincare condition}) is better known in the theory of normal forms, see, e.g., Refs. \cite{BraaksmaStolov07,Ouchi05}, or in the problem of existence of analytic invariant manifolds, see Ref.  \cite{CS14}, both for holomorphic vector fields defined near a singular point. In particular, in the problem of their local analytic linearization where much more complicated non-resonance conditions (Siegel \cite[Theorem 2.12]{Chow94} or Brjuno \cite[p. 606]{BraaksmaStolov07}) appear. 	

Returning to our problem, the linear part of $L_1=P\cdot L$ in equation (\ref{Eq. Main Eq}) can be highly degenerated and $\Lambda$ is generically the zero matrix. In fact, $\Lambda=(a_i(\00)p_j)_{i,j}$, $p_j=\d_{x_j}P(\00)$, having the diagonal matrix  $\text{diag}(\text{tr}(\Lambda),0,\dots,0)$ as canonical Jordan form, where \[\text{tr}(\Lambda)=a_1(\00)p_1+\cdots+a_d(\00)p_d=L(P)(\00).\] Thus, the only case for which $m\geq 1$, in fact, $m=1$, is when $L(P)(\00)\neq 0$. Furthermore, Poincar\'{e} condition (\ref{Poincare condition}) is satisfied if and only if \[\mu_k-nL(P)(\00)\neq 0,\quad\text{ for all } k=1,\dots,d,\, n\in\N.\] In the aforementioned papers, our situation ($m=0$ or $1$) is covered in Refs. \cite[Theorem 1.1]{Hibino99} and \cite[Theorem 1.2]{Hibino04}, claiming the solution is $(1,\dots,1)$--Gevrey while working in the variable $\xxi$, see Section \ref{Sec. Gevrey series in a germ} for definitions. For the case $m=0$, Theorem \ref{Thm Main Result} improves the divergence rate of the formal solution by showing it is $(1/k,\dots,1/k)$--Gevrey, where $k=o(P)$ is the order of $P$, see Proposition \ref{Prop. Ps sss Gevrey}. But more importantly, it identifies a possible variable to study summability phenomena. Finally, in the case $m=1$, $L(P)(\00)\neq 0$, the formal solution is convergent. In fact, by reordering the coordinates we can assume $a_1(\00) p_1\neq 0$, thus, $\xi_1=P(\xx), \xi_2=x_2,\dots, \xi_d=x_d$ is a local change of variables in which our differential operator takes the form \[P\cdot L=\xi_1\cdot \left(U(\xxi)\d_{\xi_1}+\overline{a}_2(\xxi)\d_{\xi_2}+\cdots+\overline{a}_d(\xxi)\d_{\xi_d}\right),\] where $\overline{a}_j(\xxi)=a_j(\xx)$, and $U(\xxi)=\overline{a}_1\d_{x_1}P+\cdots \overline{a}_d\d_{x_d}P$ is a unit since $U(\00)=L(P)(\00)$. Then, a standard majorant argument by working in the variable $\xi_1$ proves the convergence of the solution. We can also prove this by a slight modification in the proof of Theorem \ref{Thm Main Result}. In this way we find:

\begin{teor}\label{Thm 2} Assume the hypotheses of Theorem \ref{Thm Main Result}, but now suppose $L(P)(\00)\neq 0$. If $\mu-nL(P)(\00)I_N$ is invertible, for all $n\in\N$, then equation (\ref{Eq. Main Eq}) has a unique analytic solution at the origin $\widehat{\yy}\in \C\{\xx\}^N$.
\end{teor}

The technique to prove Theorems \ref{Thm Main Result} and \ref{Thm 2} is based on modified Nagumo norms for several variables, as introduced in Ref. \cite{CDRSS}, joint with a generalized Weierstrass division theorem that allows to write a power series as a series in the germ $P$, although the decomposition depends on the monomial order employed. Due to the compatibility of these tools, we can use a majorant series argument to establish the results.

The structure of the paper is as follows: Sections \ref{Sec. Nagumo norms} and \ref{Sec. The division algorithm} contain the technical parts of the work where we explain the properties we need on modified Nagumo norms, the Weierstrass division theorem and their compatibility. In Section \ref{Sec. Gevrey series in a germ} we recall the notions of $(s,\dots,s)$-- and $P$-$s$--Gevrey series, $s\geq 0$, and we develop some properties relating them. Sections \ref{Sec. Proof of Theorem } and \ref{Sec. Higher order} contain the proof of Theorems \ref{Thm Main Result} and \ref{Thm 2} and a simple extension to higher order systems (Corollary \ref{Coro. 1}), respectively. Finally, Section \ref{Sec. Examples} encloses several examples, including one showing that the hypotheses of Theorem \ref{Thm Main Result} are necessary to conclude the desired Gevrey type.

\section{Nagumo norms}\label{Sec. Nagumo norms}

Let us start by fixing some notation: $\N$ is the set of natural numbers including $0$, $\N^+=\N\setminus\{0\}$, and $\R{^+}$ is the set of positive real numbers. For a coordinate $t$, we will write $\frac{\d}{\d t}=\d_t$ for the corresponding derivative. If $\boldsymbol{\beta}=(\beta_1,\dots,\beta_d)\in\N^d$, we use the multi-index notation $|\boldsymbol{\beta}|=\beta_1+\cdots+\beta_d$, $\bb!=\beta_1!\cdots\beta_d!$, $\boldsymbol{x}^{\boldsymbol{\beta}}=x_1^{\beta_1}\cdots x_d^{\beta_d}$ and  $\frac{\d^{\boldsymbol{\beta}}}{\d \boldsymbol{x}^{\boldsymbol{\beta}}}=\frac{\d^{|\boldsymbol{\beta}|}}{\d x_1^{\beta_1}\cdots \d x_d^{\beta_d}}$. 

Let $d\geq 1$ be an integer. We will work with $(\C^d,\00)$ and local coordinates $\xx=(x_1,\dots,x_d)$. We also write $\xx'=(x_2,\dots,x_d)$ when removing the first variable. $\widehat{\mathcal{O}}=\C[[\xx]]$ and  $\mathcal{O}=\C\{\xx\}$ denote the rings of formal and convergent power series in $\xx$ with complex coefficients, respectively. $\mathcal{O}^\ast=\{U\in\mathcal{O} \,:\, U(\00)\neq 0\}$ will denote the corresponding groups of units. Given $\hat{f}=\sum a_\bb \xx^\bb\in\widehat{\mathcal{O}}$, $o(\hat{f})$ will denote its order: if $\hat{f}=\sum_{n=0}^\infty {f}_n$, ${f}_n=\sum_{|\bb|=n} a_\bb \xx^\bb$, is written as sum of its homogeneous components, $o(\hat{f})$ is the least integer $k$ for which ${f}_k\neq 0$. 

For $\boldsymbol{r}=(r_1,\dots,r_d)\in(\R^+)^d$, $D_{\boldsymbol{r}}=\{\xx\in \C^d \,:\, |x_j|<r_j, j=1,\dots,d\}$ is the polydisc centered at the origin with polyradius $\boldsymbol{r}$. If $r_j=r$, for all $j$, we write $D_\rr=D_r^d$ as a Cartesian product instead. By using the norm $|\xx|:=\max_{1\leq j\leq d} |x_j|$, we can write $D_r^d=\{\xx\in\C^d : |\xx|<r\}$. Also, $\mathcal{O}(D_{\boldsymbol{r}})$ and $\mathcal{O}_b(D_{\boldsymbol{r}})$ will denote the sets of holomorphic and bounded holomorphic $\C$-valued functions on the given polydisc. We denote by $J:\mathcal{O}(D_{\boldsymbol{r}})^N\rightarrow \mathcal{O}^N$ the Taylor map sending a vector function to its Taylor series at the origin.

Nagumo norms were introduced originally by M. Nagumo in Ref. \cite{Nagumo42} in his study of analytic partial differential equations. There are other alternative versions that have been used successfully in this context, see, e.g., \cite{BalserLoday09,Remy20}. We will use a variant as it appears in Ref. \cite{CDRSS} for the case of one complex variable. Let us fix two numbers $0<\rho<r$ and consider the function \[d_r(x)=\left\{\begin{array}{ll}
	r-|x|&\text{ if } |x|\geq \rho,\\
	r-\rho&\text{ if } |x|<\rho,
\end{array}\right.\] which satisfies \begin{equation}\label{Eq dx-dx}
	|d_r(x)-d_r(y)|\leq |x-y|,\quad x,y\in D_r.
\end{equation} The number $\rho$ can be chosen arbitrarily but we will always take $\rho=r/2$.

Fix a polyradius $\rr=(r_1,\dots,r_d)$. If $f\in\mathcal{O}(D_\rr)$ and $m\in\N$, we consider the family of \textit{Nagumo norms} \begin{equation}\label{Nagumo n} \|f\|_m:=\sup_{\xx\in D_\rr} |f(\xx)|d_{r_1}(x_1)^m\cdots d_{r_d}(x_d)^m.
\end{equation} These norms depend on $\rr$, but to simplify notation we omit this dependence. There is no reason for these values to be finite, for instance, if $m=0$ this norm reduced to the maximum norm. Note that if $\|f\|_k$ is finite and $m>k$, then \begin{equation}\label{Eq fm f0}
	\|f\|_m\leq (r_1/2)^{m-k}\cdots (r_d/2)^{m-k} \|f\|_k.
\end{equation} In particular, if $k=0$, i.e., if $f\in\mathcal{O}_b(D_\rr)$, all its Nagumo norms are finite.

We collect in the next proposition the main properties of these norms we will use in the proof of Theorems \ref{Thm Main Result} and \ref{Thm 2}, including their behavior under the shift operators \begin{equation}\label{Eq. Shifts}
	S_j(f)(\xx)=\left\{\begin{array}{ll}
		\left(f(\xx)-f(x_1,\dots,x_{j-1},0,x_{j+1},\dots,x_d)\right)/x_j&\text{ if } x_j\neq 0,\\
		\frac{\d f}{\d{x_j}}(\xx)&\text{ if } x_j=0.
	\end{array}\right.
\end{equation}

\

\begin{prop}\label{Prop Nagumo norms} Fix $m,k\in \N$ and $1\leq j\leq d$. If $f,g\in\mathcal{O}(D_\rr)$, then
	
	\begin{enumerate}[label=(\roman*)]
		\item $\|f+g\|_m\leq \|f\|_m+\|g\|_m$ and $\|fg\|_{m+k}\leq \|f\|_m\|g\|_k$. 
		\item $\left\|\frac{\d f}{\d x_j}\right\|_{m+1}\leq e(m+1)\prod_{i\neq j} (r_i/2)\|f\|_m$.
		\item $\|S_j(f)\|_m\leq \frac{4}{r_j}\|f\|_m$.
	\end{enumerate}
	
\end{prop}

\begin{proof} The inequalities in (i) are clear from the definition. We prove (ii) and (iii) for the variable $x_1$. If $\xx=(x_1,\xx')\in D_\rr$, we have \begin{equation}\label{Eq f fm} |f(\xx)|d_{r_1}(x_1)^m\cdots d_{r_d}(x_d)^m\leq \|f\|_m.
	\end{equation} To establish (ii), we use Cauchy's formula \begin{equation}\label{Eq. f' f}
		\left|\frac{\d f}{\d x_1}(\xx)\right|=\left|\frac{1}{2\pi}\int_{|\xi-x_1|=R}\frac{f(\xi,\xx')}{(\xi-x_1)^2}d\xi\right|\leq \frac{1}{R}\sup_{|\xi-x_1|=R} |f(\xi,\xx')|,
	\end{equation} valid for $0<R<r-|x_1|$. If $|\xi-x_1|=R$, then $d_{r_1}(x_1)-R\leq d_{r_1}(\xi)$ by applying inequality (\ref{Eq dx-dx}). In particular, if $0<R<d_{r_1}(x_1)$ we find \[|f(\xi,\xx')|\leq \|f\|_m d_{r_2}(x_2)^{-m}\cdots d_{r_d}(x_d)^{-m} (d_{r_1}(x_1)-R)^{-m}.\] For the case $m>0$, choose $R=\frac{d_{r_1}(x_1)}{m+1}$ to find  \[\left|\frac{\d f}{\d x_1}(\xx)\right|\leq \frac{(m+1)\|f\|_m}{d_{r_1}(x_1)^{m+1}d_{r_2}(x_2)^m\cdots d_{r_d}(x_d)^m}\left(1+\frac{1}{m}\right)^m.\] Therefore, using the inequality $(1+1/m)^m<e$ we conclude that \begin{align*}
		\left|\frac{\d f}{\d x_1}(\xx)d_{r_1}(x_1)^{m+1}\cdots d_{r_d}(x_d)^{m+1}\right|&\leq e(m+1) d_{r_2}(x_2)\cdots d_{r_d}(x_d)\|f\|_m\\
		&\leq e(m+1)\left(\frac{r_2}{2}\cdots \frac{r_d}{2}\right)\|f\|_m,
	\end{align*} as we wanted to show. Now, if $m=0$, taking $R=d_{r_1}(x_1)/e$ in (\ref{Eq. f' f}) we find \[\left|\frac{\d f}{\d x_1}(\xx)d_{r_1}(x_1)\cdots d_{r_d}(x_d)\right|\leq e\left(\frac{r_2}{2}\cdots \frac{r_d}{2}\right)\|f\|_0,\] as desired.	
	
	For (iii), by inequality (\ref{Eq f fm}), $|f(0,\xx')|$ is less than 
	\[\|f\|_m ((r_1/2)d_{r_2}(x_2)\cdots d_{r_d}(x_d))^{-m}\\
	\leq \|f\|_m (d_{r_1}(x_1)d_{r_2}(x_2)\cdots d_{r_d}(x_d))^{-m},\] for all $\xx\in D_\rr$. Hence, if $|x_1|\geq r_1/2$, \[\left|\frac{f(x_1,\xx')-f(0,\xx')}{x_1}\right|\leq \frac{4}{r_1}\|f\|_m d_{r_1}(x_1)^{-m}\cdots d_{r_d}(x_d)^{-m}.\] For $|x_1|<r_1/2$ we can use the maximum modulus principle and the estimate above to see that \[|S_1(f)(\xx)|\leq \max_{|\xi|=r_1/2} |S_1(f)(\xi,\xx')|\leq \frac{4}{r_1}\|f\|_m ((r_1/2)d_{r_2}(x_2)\cdots d_{r_d}(x_d))^{-m}.\] Since $d_{r_1}(x_1)=r_1/2$ if $|x_1|<r_1/2$, we find in all cases that $|S_1(f)(\xx)d_{r_1}(x_1)^m\cdots d_{r_d}(x_d)^m|\leq (4/r_1) \|f\|_m$ as required.
\end{proof}

For vector--valued $\boldsymbol{y}=(y_1,\dots,y_N)\in\mathcal{O}(D_{\rr})^N$, and matrix--valued $A=(A_{i,j})\in\mathcal{O}(D_{\rr})^{N\times N}$ maps, we extend Nagumo norms by the rules \begin{equation}\label{Eq. Nagumo N}
	\|\yy\|_m:=\max_{1\leq i\leq N} \|y_i\|_m,\quad \|A\|_m:=\max_{1\leq i\leq N} \sum_{j=1}^N \|A_{i,j}\|_m.
\end{equation} Then, it is immediate to check that \[\|f\cdot\yy\|_{m+k}\leq \|f\|_m\|\yy\|_k,\quad \|A\cdot\yy\|_{m+k}\leq \|A\|_m\|\yy\|_k,\] \[\|A\cdot B\|_{m+k}\leq \|A\|_m\|B\|_k,\] for all $f\in\mathcal{O}(D_\rr)$, $\yy\in\mathcal{O}(D_\rr)^N$, and $A,B\in\mathcal{O}(D_\rr)^{N\times N}$.

\section{The division algorithm}\label{Sec. The division algorithm}

We recall here a generalized Weierstrass division theorem by following closely Ref. \cite{MS19}, and whose original version is due to J.M. Aroca, H. Hironaka, and J. L. Vicente, see Ref. \cite{AHV75}. For the sake of completeness we include the proof for convergent series including the compatibility of the division algorithm with the Nagumo norms introduced in Section \ref{Sec. Nagumo norms}.

We will use the partial order $\leq$ on $\N^d$ defined by $\aa\leq \bb$ if  $\a_j\leq \beta_j$, for all $j=1,\dots, d$. Thus $\aa\not\leq \bb$ means there is an index $j$ such that $\beta_j<\a_j$. We also use the notation \[\Delta_\aa:=\left\{ \sum g_\bb \xx^{\bb}\in\widehat{\mathcal{O}} : g_\bb=0 \text{ if } \aa\leq \bb\right\}.\]

Given $\boldsymbol{\a}=(\a_1,\dots,\a_d)\in\N^{d}\setminus \{\00\}$, a power series  $\hat{f}=\sum_{\bb\in\N^d} f_\bb \xx^\bb\in \widehat{\mathcal{O}}$ can be written uniquely as a series in the monomial $\xx^\aa$ as \begin{equation}\label{Eq decomp f xa}
	\hat{f}=\sum_{n=0}^\infty \hat{f}_{\aa,n}(\boldsymbol{x})\boldsymbol{x}^{n\boldsymbol{\a}},\quad \hat{f}_{\aa,n}(\xx)=\sum_{\aa\not\leq\bb}f_{n\boldsymbol{\a}+\boldsymbol{\beta}}\boldsymbol{x}^{\boldsymbol{\beta}}\in \Delta_\aa.
\end{equation} This decomposition is obtained by a repeated use of the canonical division algorithm by $\xx^\aa$: given $\hat{f}\in\widehat{\mathcal{O}}$, there are unique $q\in\widehat{\mathcal{O}}$, $r\in \Delta_\aa$ such that \[\hat{f}=q \xx^\aa+r,\quad \text{ where } \quad q=\sum_{\aa\leq \bb} f_\bb x^{\bb-\aa},\quad  r=\sum_{\aa\not\leq\bb} f_\bb \xx^\bb.\] Moreover, if $f\in\mathcal{O}(D_\rr)$, then $q,r\in\mathcal{O}(D_\rr)$. Actually, we can use the shift operators (\ref{Eq. Shifts}) to write \[q=Q_\aa(f):=S_1^{\a_1}\circ \cdots \circ S_d^{\a_d}(f),\quad r=R_\aa(f):=f-Q_\aa(f)\cdot \xx^{\aa}.\] In particular, Proposition \ref{Prop Nagumo norms} (iii) shows that \begin{equation}\label{Eq Qa Ra}
	\|Q_\aa(f)\|_m\leq  \frac{4^{|\aa|}}{\rr^\aa}\|f\|_m,\quad \|R_\aa(f)\|_m\leq(1+4^{|\aa|})\|f\|_m,
\end{equation} for all $m\in\N$. By taking $m=0$ we conclude $Q_\aa, R_\aa: \mathcal{O}_b(D_\rr)\to \mathcal{O}_b(D_\rr)$ are linear continuous maps.

The generalized Weierstrass division allows to extend the previous considerations by dividing by an element of $\widehat{\mathcal{O}}\setminus\{0\}$ with zero constant term, but not in a canonical way. We will focus on division by an analytic germ  $\PP \in \mathcal{O}\setminus\{0\}$, $\PP(\00)= 0$. The division is determined by $P$ and an injective linear form $\ell:\N^d\rightarrow\R^+$,  $\ell(\aa)=\ell_1\a_1+\cdots+\ell_d\a_d$ used to order the monomials by the rule  \[\xx^\aa<_\ell\xx^\bb\quad \text{ if }\quad \ell(\aa)<\ell(\bb).\] Then, any $\hat{f}=\sum f_\bb \xx^\bb\in\widehat{\mathcal{O}}\setminus\{0\}$, has a \textit{minimal exponent} $\nu_\ell(\hat{f})$ with respect to $\ell$, i.e., $\nu_\ell(\hat{f})=\aa$ where $\xx^\aa=\min_\ell\{\xx^\bb : f_\bb\neq 0\}$, and the minimum is taken according to $<_\ell$. The division process, for formal and convergent series, can be stated as follows, c.f., Ref. \cite[Lemmas 2.4, 2.6]{MS19}.

\begin{prop}[Generalized Weierstrass Division] \label{Generalized Weierstrass Division} Let $P$ and $\ell$ as above. For every $\hat{g}\in\widehat{\mathcal{O}}$, there are unique $\hat{q}\in\widehat{\mathcal{O}}$, $\hat{r}\in \Delta_{\nu_\ell(P)}$ such that \[\hat{g}=\hat{q}P+\hat{r}.\] Moreover, if $\rho>0$ is sufficiently small, then for every $g\in\mathcal{O}_b(D_{\rho(\ell)})$, $\rho(\ell)=(\rho^{\ell_1},\dots,\rho^{\ell_d})$, there are unique $q\in \mathcal{O}_b(D_{\rho(\ell)})$, $r\in \mathcal{O}_b(D_{\rho(\ell)})$ with $J(r)\in \Delta_{\nu_\ell(P)}$ such that \[g=qP+r,\quad Q_{P,\ell}(g):=q,\quad  R_{P,\ell}(g):=r.\] The corresponding operators $Q_{P,\ell}, R_{P,\ell}:\mathcal{O}_b(D_{\rho(\ell)})\rightarrow \mathcal{O}_b(D_{\rho(\ell)})$ are linear and continuous. In fact, if $\rho$ is sufficiently small, then \[\|Q_{P,\ell}(g)\|_m\leq \frac{2\cdot 4^{|\nu_\ell(P)|}}{\rho^{\ell(\nu_\ell(P))}}\|g\|_m,\quad \|R_{P,\ell}(g)\|_m\leq 2(1+4^{|\nu_\ell(P)|})\|g\|_m,\] for all  $m\in\N$.
\end{prop}

\begin{proof} By the choice of $\rho(\ell)$, we have $|\xx^\bb|\leq \rho^{\ell(\bb)}$ if $\xx\in D_{\rho(\ell)}$. Let us write $\aa=\nu_\ell(P)$. Without loss of generality we can assume $P=\xx^\aa+\widetilde{P}$, where $\widetilde{P}\in\mathcal{O}\setminus\{0\}$ and $\nu_\ell(\widetilde{P})>_\ell \xx^\aa$. Then, solving $g=qP+r$ or $q\xx^\aa+r=g-q\widetilde{P}$ for $q$ and $r$ is equivalent to find a fixed point for  \begin{equation}\label{Eq Aux q}
		q=Q_\aa(g-q\widetilde{P}),
	\end{equation} joint with $r=R_\aa(g-q\widetilde{P})$. Note that if $\rho$ is sufficiently small, we can choose a constant $K>0$ such that  $\|\widetilde{P}\|_0\leq K\rho^{\ell(\nu_\ell(\widetilde{P}))}$. Consider the map $\phi_g:\mathcal{O}_b(D_{\rho(\ell)})\to \mathcal{O}_b(D_{\rho(\ell)})$ given by $\phi_g(h)=Q_\aa(g-h\widetilde{P})$. By using the first inequality in (\ref{Eq Qa Ra}) for $m=0$ we see that \[\|\phi_g(h_1)-\phi_g(h_2)\|_0=\left\|Q_\aa((h_1-h_2)\widetilde{P})\right\|_0\leq 4^{|\aa|}K\rho^{\ell(\nu(\widetilde{P}))-\ell(\aa)}\|h_1-h_2\|_0.\] Thus, $\phi_g$ defines a contraction if $\rho$ is small enough, i.e., if  $4^{|\aa|}K\rho^{\ell(\nu_\ell(\widetilde{P}))-\ell(\aa)}<1$, and it has a unique fixed point $q$. This determines the existence and uniqueness of $q$ and $r$. Finally, from equation (\ref{Eq Aux q}) we find that \[\|Q_{P,\ell}(g)\|_m\leq \frac{4^{|\aa|}/\rho^{\ell(\aa)}}{1-4^{|\aa|}K\rho^{\ell(\nu_\ell(\widetilde{P}))-\ell(\aa)}}\|g\|_m,\] and from   $r=R_\aa(g-q\widetilde{P})$, that \begin{align*}
		\|R_{P,\ell}(g)\|_m&\leq (1+4^{|\aa|})\left(\! 1+\frac{\|\widetilde{P}\|_04^{|\aa|}/\rho^{\ell(\aa)}}{1-4^{|\aa|}K\rho^{\ell(\nu_\ell(\widetilde{P}))-\ell(\aa)}}\!\right)\!\|g\|_m\\
		&=\frac{(1+4^{|\aa|})\|g\|_m}{1-4^{|\aa|}K\rho^{\ell(\nu_\ell(\widetilde{P}))-\ell(\aa)}}.
	\end{align*} Therefore, the result follows by taking additionally $\rho>0$ such that $4^{|\aa|}K\rho^{\ell(\nu_\ell(\widetilde{P}))-\ell(\aa)}<1/2$.
\end{proof}

By a repeated application of the previous proposition, see Ref. \cite[Corollary 2.5]{MS19}, any $\hat{f}\in\widehat{\mathcal{O}}$ can be written uniquely as \begin{equation}\label{Eq. Decomposition formal}
	\hat{f}=\sum_{n=0}^\infty \hat{f}_{P,\ell,n} P^n,\quad \hat{f}_{P,\ell,n}\in \Delta_{\nu_\ell(P)}.
\end{equation} For the convergent case, we have a similar result, see Ref. \cite[Corollary 2.7]{MS19}.

\begin{coro}\label{Coro Decomposition convergent}
	If $s>0$ is such that the operators $Q_{P,\ell}$ and $R_{P,\ell}$ are defined over $\mathcal{O}_b(D_{s(\ell)})$, there is $r=r(s)>0$, depending only on $s$, such that for any $f\in\mathcal{O}_b(D_{s(\ell)})$ we can find a unique sequence $(f_n)_{n\in\N}\subset \mathcal{O}_b(D_r^d)$ with $J(f_n)\in \Delta_{\nu_\ell(P)}$, such that \begin{equation}\label{Eq. Decomposition convergent}
		f=\sum_{n=0}^\infty f_n \PP^n,\quad  f_n=R_{P,\ell}\circ Q_{P,\ell}^n(f),\quad \text{both convergent for } |\xx|<r.
	\end{equation} 
\end{coro}

\begin{proof} By applying Proposition \ref{Generalized Weierstrass Division} we obtain \[f=\sum_{n=0}^{N-1} R_{P,\ell}(Q_{P,\ell}^n(f))P^n+Q^N(f) P^N,\quad \text{ for all } N\in\N.\] If we choose $0<r\leq s$ with $M=\sup_{|\xx|<r} |P(\xx)|<s^{\ell(\nu_\ell(P))}/(2\cdot 4^{|\nu_\ell(P)|})=1/b$, then we can estimate  \[\sup_{|\xx|<r} \left|f(\xx)-\sum_{n=0}^{N-1} R_{P,\ell}(Q_{P,\ell}^n(f))(\xx)P(\xx)^n\right|\leq (bM)^N \sup_{\yy\in D_{s(\ell)}} |f(\yy)|.\] The result follows by taking $N\to+\infty$.
\end{proof}

To finish this section we remark that if $\hat{f}_l=\sum_{n=0}^\infty  f_{l,n} P^n$, $l=1,\dots,m$, where $f_{l,n}\in \mathcal{O}_b(D_r^d)$ for a common $r$, decomposition (\ref{Eq. Decomposition formal}) for their product is given by \begin{align}
	\nonumber \hat{f}_1\cdots\hat{f}_m=&\sum_{j_1,\dots,j_m\geq 0} f_{1,j_1}\cdots f_{m,j_m} P^{j_1+\cdots+j_m}\\
	\nonumber =&\sum_{k,j_1,\dots,j_m\geq 0} R_{P,\ell}(Q^k_{P,\ell}(f_{1,j_1}\cdots f_{m,j_m})) P^{k+j_1+\cdots+j_m}\\
	\label{Eq. ProductS}=&\sum_{n=0}^\infty \left[\sum_{{k+j_1+\cdots+j_m=n}\atop {k,j_1,\dots,j_m\geq 0}} R_{P,\ell}(Q^k_{P,\ell}(f_{1,j_1}\cdots f_{m,j_m}))\right] P^{n}.
\end{align} In particular, note that for $n=0$ we find that  \begin{equation}\label{Eq. RRR} R_{P,\ell}(\hat{f}_1\cdots \hat{f}_m)=R_{P,\ell}(f_{1,0}\cdots f_{m,0}).\end{equation} Also,  note that for two factors the previous sum takes the form \begin{align}\label{Eq. Product}
	\hat{f}_1\cdot \hat{f}_2=\sum_{n=0}^\infty \Big(\sum_{k=0}^n\sum_{j=0}^k R_{P,\ell} (Q_{P,\ell}^{n-k}(f_{1,j} f_{2,k-j}))\Big)P^n.
\end{align}

\section{Gevrey series}\label{Sec. Gevrey series in a germ}

Given $\boldsymbol{s}=(s_1,\dots,s_d)\in\R_{\geq 0}^d$ and $\hat{f}=\sum_{\bb\in\N^d} a_\bb \xx^\bb\in\widehat{\mathcal{O}}$, the series $\hat{f}$ is said to be a \textit{$\boldsymbol{s}$--Gevrey} if we can find $C,A>0$ such that $|a_\bb|\leq CA^{|\bb|} \bb!^{\boldsymbol{s}}$, for all $\bb\in\N^d$. Note that $\boldsymbol{s}=\00$ means convergence. We will be interested in the case $s_1=\dots=s_d=s\geq 0$. Thanks to the inequalities \[\bb!\leq |\bb|!\leq d^{|\bb|}\bb!,\] a series  $\hat{f}$ is $(s,\dots,s)$--Gevrey if and only if there are $C,A>0$ such that \[|a_\bb|\leq CA^{|\bb|} |\bb|!^s,\quad \bb\in\N^d.\] We denote by $\widehat{\mathcal{O}}_{s}$ the set of $(s,\dots,s)$--Gevrey series.

\begin{nota} For any $s\geq0$, $\widehat{\mathcal{O}}_{s}$ is closed under sums, products, partial derivatives, composition, and it contains $\mathcal{O}$. This properties can be seen as a particular case in the setting of ultradifferentiable functions. In that framework, the Gevrey sequence $(n!^s)_{n\in\N}$ is generalized by a sequence of positive numbers $(M_n)_{n\in\N}$ satisfying log-convexity ($M_n^2\leq M_{n-1}M_{n+1}$), stability under derivatives ($M_{n+1}\leq K^nM_n$ for some $K>0$) and the condition $M_n^{1/n}\to +\infty$ as $n\to\infty$, see e.g., Refs.  \cite{Thilliez20008,RainerG2016} including other stability properties in a more general context.
\end{nota}

According to the previous remark, if $\hat{f}\in\widehat{\mathcal{O}}_{s}$, the same is true for $\hat{f}(A\xx)$, for all matrices $A\in\C^{d\times d}$, c.f., Ref. \cite[Lemma 2.1]{Hibino99}. In particular, we highlight the following simple statement we will need later.

\begin{lema}\label{Lemma linear change} Let $s\geq 0$. Then, $\hat{f}\in\widehat{\mathcal{O}}_{s}$ if and only if $\hat{f}(A\xx)\in \widehat{\mathcal{O}}_{s}$, for all $A\in\text{GL}_d(\C)$.
\end{lema}

Consider a germ $P\in\mathcal{O}\setminus\{0\}$ such that $P(\00)=0$, and $s\geq 0$. There are equivalent definitions for Gevrey series with respect to the germ $P$, see Ref.   \cite[Definition/Proposition 7.5]{MS19}. For simplicity, we will use the characterization given in Ref. \cite[Lemma 4.1]{CMS19}.

\begin{defi} A series $\hat{f}\in\widehat{\mathcal{O}}$ is \textit{$P$-$s$--Gevrey series} if there is a polyradius $\rr$, constants $C,A>0$ and a sequence $\{f_n\}_{n\in\N}\in\mathcal{O}_b(D_\rr)$ such that \begin{equation}\label{Def. P s Gevrey}
		\hat{f}=\sum_{n=0}^\infty f_n P^n,\quad \text{ where }  \sup_{\xx\in D_\rr}|f_n(\xx)|\leq CA^n n!^s.
	\end{equation} We will use the notation $\widehat{\mathcal{O}}^{P,s}$ for the set of $P$-$s$--Gevrey series. 
\end{defi}

This generalizes the notion of $s$--Gevrey series in $x_j$, uniformly in the other variables ($x_j$-$s$--Gevrey series in our notation). In fact, setting $j=1$ to fix ideas, the classical notion requires that when we write $\hat{f}=\sum_{n=0}^\infty f_{n} x_1^n$ as a power series in $x_1$, there is a polyradius $\rr'\in \R^{d-1}$ such that $f_n\in\mathcal{O}_b(D_{\rr'})$ and $\sup_{\xx'\in D_{\rr'}}|f_n(\xx')|\leq CA^n n!^{s},$ for adequate constants $C,A$.

By using the generalized Weierstrass division we can show the notion of $P$-$s$--Gevrey series is well-defined, in the sense that it is independent of the decomposition (\ref{Def. P s Gevrey}). Note it is enough to check the definition for the decomposition (\ref{Eq. Decomposition formal}) induced by a given injective linear form $\ell:\N^d\rightarrow \R^+$. In fact, if (\ref{Def. P s Gevrey}) holds, and since all $f_n$ are defined in a common polydisc, we can use decomposition (\ref{Eq. Decomposition convergent}) to find $\rho>0$ and sequences $\{f_{n,j}\}_{n\in\N}\subset \mathcal{O}_b(D_\rho^d)$ with $J(f_{n,j})\in\Delta_\ell(P)$, such that $f_{n}=\sum_{j=0}^\infty f_{n,j} P^j$, valid for $|\xx|<\rho$, where $f_{n,j=}R_{P,\ell}\circ Q_{P,\ell}^j(f_n)$. Therefore, the decomposition (\ref{Eq. Decomposition formal}) of $\hat{f}$ is given by \[\hat{f}=\sum_{n=0}^\infty g_n P^n,\quad g_n=\sum_{j=0}^n f_{j,n-j}\in \mathcal{O}_b(D_\rho^d),\quad J(g_n)\in \Delta_\ell(P),\] and the sequence $(g_n)_{n\in\N}$ exhibits $s$--Gevrey bounds since \[|g_n(\xx)|\leq \sum_{j=0}^n \|R_{\PP,\ell}\| \|Q_{\PP.\ell}\|^{n-j} \sup_{|\yy|\leq \rho} |f_j(\yy)|\leq \sum_{j=0}^n \|R_{\PP,\ell}\| \|Q_{\PP.\ell}\|^{n-j} CA^j j!^s,\] for $|\xx|<\rho$, as we wanted to show. 

From the previous definition it is easy to deduce many properties on this type of series. We recall the following, valid for $s\geq0$ and $P,Q\in\mathcal{O}\setminus\{0\}$ such that $P(\00)=Q(\00)=0$, c.f., Ref. \cite[Corollary 4.2, Lemma 4.3]{CMS19}: 

\begin{enumerate}[label=(\roman*)]
	\item $\widehat{\mathcal{O}}^{P,s}$ is stable under sums, products and partial derivatives.
	
	\item $\mathcal{O}\subset \widehat{\mathcal{O}}^{P,s}$.
	
	\item For any $k\in\N^+$, $\widehat{\mathcal{O}}^{P^k,ks}=\widehat{\mathcal{O}}^{P,s}$.
	
	\item If $Q$ divides $P$, then $\widehat{\mathcal{O}}^{P,s}\subseteq \widehat{\mathcal{O}}^{Q,s}$. In particular, if $Q=U\cdot P$, $U\in\mathcal{O}^\ast$, then $\widehat{\mathcal{O}}^{P,s}=\widehat{\mathcal{O}}^{Q,s}$. 
	
	\item Let $\phi:(\C^d,\00)\to(\C^d,\00)$ be analytic, $\phi(\00)=\00$, and assume $P\circ \phi$ is not identically zero. If $\hat{f}\in \widehat{\mathcal{O}}^{P,s}$, then $\hat{f}\circ \phi\in \widehat{\mathcal{O}}^{P\circ \phi,s}$.
	
	\item If $P(\xx)=\xx^\aa$, $\aa\in\N^d\setminus\{\00\}$, then $\hat{f}=\sum f_{\boldsymbol{\beta}}\boldsymbol{x}^{\boldsymbol{\beta}}\in \widehat{\mathcal{O}}^{\xx^\aa,s}$ if and only if there are constants $C,A>0$ satisfying \begin{equation}\label{Eq. Gevrey bounds monomials}
		|f_{\boldsymbol{\beta}}|\leq CA^{|\boldsymbol{\beta}|}\min\{ \beta_j!^{s/\a_j} : j=1,\dots,d, \a_j\neq 0\} ,\quad \boldsymbol{\beta}\in\N^d.
	\end{equation}  
\end{enumerate}

\

It follows from (\ref{Eq. Gevrey bounds monomials}) that if $\hat{f}\in \widehat{\mathcal{O}}^{\xx^\aa,s}$, then $\hat{f}\in\widehat{\mathcal{O}}_{s/|\aa|}$. Indeed, this is a consequence of the inequality 
$\min\{a_1,\dots,a_d\}\leq a_1^{\tau_1}\cdots a_d^{\tau_d}$, valid for all $a_j>0$ and $\tau_j\geq 0$ such that $\tau_1+\cdots+\tau_d=1$, by applying it to $\tau_j=\a_j/|\aa|$. This property can be generalized to an arbitrary germ and we have the following new inclusion of rings of Gevrey series.

\begin{prop}\label{Prop. Ps sss Gevrey} Consider $P\in \mathcal{O}$ with $o(P)=k\geq 1$. Then, a $P$-$s$--Gevrey series is a $(s/k,\dots,s/k)$--Gevrey series. In symbols,  \[\widehat{\mathcal{O}}^{P,s}\subseteq \widehat{\mathcal{O}}_{s/k}.\]
\end{prop}

\begin{proof} Write $P=\sum_{j=k}^\infty P_j$ as sum of homogeneous polynomials. Since $P_k\neq 0$, we can find $\boldsymbol{a}\neq \00$ such that $P_k(\boldsymbol{a})\neq 0$. Choose $A\in\text{GL}_n(\C)$ having $\boldsymbol{a}$ as first column. If we set $Q(\xx)=P(A\xx)$ and we write it as sum of its homogeneous components $Q=\sum Q_j$, then $Q_j(\xx)=P_j(A\xx)$, and $Q_k(\xx)=P_k(\boldsymbol{a})x_1^k+\cdots$, i.e., $o(Q)=k$ and $Q_k(1,0,\dots,0)\neq 0$.
	
	Consider a $P$-$s$--Gevrey series $\hat{f}$. Then $\hat{f}_0(\xx)=\hat{f}(A\xx)=\sum a_\bb \xx^\bb$ is a $Q$-$s$--Gevrey series, thanks to (v) above. We consider the change of variables \begin{equation}\label{Eq. Blow up}
		x_1=z_1,\quad x_2=z_1z_2,\quad \dots, x_d=z_1z_d,
	\end{equation} that geometrically corresponds to a local expression for the blow--up of the origin in $\C^d$, see, e.g., Ref. \cite{MM80}. If $R(\zz)=Q(\xx)$ and $\hat{f}_1(\zz)=\hat{f}_0(\xx)$, we see $\hat{f}_1$ is a $R$-$s$--Gevrey series, again by (v). On the one hand, \[R(\zz)=Q(z_1,z_1z_2,\dots,z_1z_d)=\sum_{j=k}^\infty z_1^j Q_j(1,z_2,\dots,z_d)=z_1^k U(\zz),\] where $U$ is a unit, since $U(\00)=Q_k(1,0,\dots,0)\neq 0$. Thus, we conclude $\hat{f}_1$ is $z_1^k$-$s$--Gevrey, or equivalently, a $z_1$-$s/k$--Gevrey series. Now, since  \[
	\hat{f}_1(\zz)=\sum_{\bb\in\N^d} a_\bb z_1^{|\bb|} z_2^{\beta_2}\cdots z_d^{\beta_d}\\
	=\sum_{{(n,\boldsymbol{\gamma})\in\N\times\N^{d-1}}\atop {n\geq |\boldsymbol{\gamma}|}} a_{n-|\boldsymbol{\gamma}|,\boldsymbol{\gamma}} z_1^n \zz'^{\boldsymbol{\gamma}},\] we can find constants $C,A>0$ such that $|a_{n-|\boldsymbol{\gamma}|,\boldsymbol{\gamma}}|\leq CA^{n+|\boldsymbol{\gamma}|}n!^{s/k}$. Therefore, in the index $\bb=(n,\boldsymbol{\gamma})$, we find the bound \[|a_\bb|\leq CA^{\beta_1+2\beta_2+\cdots+2\beta_d} |\bb|!^{s/k},\quad \text{ for all } \bb\in\N^d.\] This means $\hat{f}_0$ and $\hat{f}$ are $(s/k,\dots,s/k)$--Gevrey series, due to Lemma \ref{Lemma linear change}.
\end{proof}

\begin{nota} Proposition \ref{Prop. Ps sss Gevrey} and Lemma \ref{Lemma linear change} show that if $\hat{f}\in \widehat{\mathcal{O}}^{\xx^\aa,s}$, then $\hat{f}(A\xx)\in\widehat{\mathcal{O}}_{s/|\aa|}$, for all $A\in \C^{d\times d}$. However, being $\xx^\aa$-$s$--Gevrey is not stable under linear changes of variable. We illustrate this by a simple example: the series $\hat{f}(x_1,x_2)=\sum_{n=0}^\infty n! (x_1x_2)^n$ is $x_1x_2$-$1$--Gevrey, but  \[\hat{f}_0(\xi_1,\xi_2)=\hat{f}(\xi_1+\xi_2,\xi_1-\xi_2)=\sum_{j,k\geq 0} \binom{j+k}{j} (j+k)! (-1)^k \xi_1^{2j} \xi_2^{2k},\] is $(1/2,1/2)$--Gevrey in $\xi_1,\xi_2$, but not $\xi_1\xi_2$-$1$--Gevrey.
\end{nota}

\section{Proof of the main results}\label{Sec. Proof of Theorem }

The idea of both proofs is to find a formal solution of the form  \begin{equation}\label{Eq. Proof y}
	\widehat{\yy}=\sum_{n=0}^\infty y_n P^n,
\end{equation} according to decomposition (\ref{Eq. Decomposition formal}) associated to an injective linear form $\ell:\N^d\to\R^+$, $\ell(\aa)=\ell_1\a_1+\cdots+\ell_d\a_d$ that we fix from now on. Then we find recursively the coefficients $y_n$ such that $J(y_n)\in \Delta_{\nu_\ell(P)}^N$, and use majorant series employing the Nagumo norms to establish the derides bounds. 

It is important to remark that under the conditions of Theorems \ref{Thm Main Result} and \ref{Thm 2}, equation (\ref{Eq. Main Eq}) admits a unique formal power series solution $\widehat{\yy}\in\C[[\xx]]^N$ such that $\widehat{\yy}(\00)=\00$. This can be seen directly by writing $\widehat{\yy}$ as the sum of its homogeneous components in $\xx$ and then finding these terms recursively after plugging $\widehat{\yy}$ in (\ref{Eq. Main Eq}). Hence, if we are able to find a formal solution of (\ref{Eq. Main Eq}) of the form (\ref{Eq. Proof y}), it coincides with the unique formal power series solution of the problem.

\begin{proof}[Proof of Theorem \ref{Thm Main Result}]
	
	We divide it in several steps. Before starting we note that if $F(\xx,\00)\equiv 0$, the unique formal power series solution if the zero series. Thus we assume $c(\xx):=F(\xx,\00)\not\equiv 0$.
	
	\textbf{Step 0} (Preliminaries) Since $F$ is analytic and $\mu=\frac{\d F}{\d \yy}(\00,\00)$, we can write it as a convergent power series in $\yy$, say  \[F(\xx,\yy)=c(\xx)+(\mu+A(\xx))\yy+\sum_{|I|\geq 2}A_I(\xx) \yy^I,\] where $c,A_I\in \mathcal{O}_b(D_{r'} ^d)^N$, $A\in \mathcal{O}_b(D_{r'}^d)^{N\times N}$, $A(\00)=\00$, and the summation is taken over all $I=(i_1,\dots,i_N)\in\N^N$ such that $|I|=i_1+\cdots+i_N\geq 2$. Furthermore, we can find $K,\delta>0$ such that \begin{equation}\label{Eq. Bound F K d}
		\|A_I\|_0\leq K \delta^{|I|},\quad \text{ for all } I\in \N^N.
	\end{equation} Note that all the previous coefficients are defined in the common polydisc of polyradius $(r',\dots,r')$. By reducing $r'$ if necessary, we can also assume that the coefficients of $L=a_1\d_{x_1}+\cdots+a_d \d_{x_d}$ in (\ref{Eq. Main Eq}) -which are not all identically zero- belong to $\mathcal{O}_b(D_{r'}^d)$.

	Given the injective linear form $\ell:\N^d\to\R^+$, take  $\rho'>0$ such that $Q_{P,\ell},R_{P,\ell}:\mathcal{O}_b(D_{\rho'(\ell)})\to\mathcal{O}_b(D_{\rho'(\ell)})$ are defined (Proposition \ref{Generalized Weierstrass Division}). We consider the linear operators $Q,R:\mathcal{O}_b(D_{\rho'(\ell)})^N\to \mathcal{O}_b(D_{\rho'(\ell)})^N$ given by $$Q(f_1,\dots,f_N)=(Q_{P,\ell}(f_1),\dots,Q_{P,\ell}(f_N)), \qquad R(f_1,\dots,f_N)=(R_{P,\ell}(f_1),\dots,R_{P,\ell}(f_N)).$$ Then, by using the norms (\ref{Eq. Nagumo N}) we see that \begin{equation}\label{Eq. Q R proof}
		\|Q(\boldsymbol{f})\|_m\leq \|Q\|\cdot \|\boldsymbol{f}\|_m,\quad \|R(\boldsymbol{f})\|_m\leq \|R\|\cdot \|\boldsymbol{f}\|_m,
	\end{equation} for all $m\in\N$ and $\boldsymbol{f}\in\mathcal{O}_b(D_{\rho'(\ell)})^N$, where to simplify notation we write $\|Q\|=2\cdot 4^{|\nu_\ell(P)|}/\rho'^{\ell(\nu_\ell(P))}$ and $\|R\|=2(1+4^{|\nu_\ell(P)|})$, according to the values found in Proposition \ref{Generalized Weierstrass Division}. The same considerations and inequalities are valid for matrix-valued maps. It will be important for later to note that $\|R\|$ is independent of the radius, since we will shrink $\rho'$ during this proof.

	Now, choose $0<s<\rho'$ with $s^{l_j}\leq r'$, for all $j$, in order to apply Corollary \ref{Coro Decomposition convergent} to the previous functions. Then, we conclude there is $r>0$ such that the maps can be written as \begin{equation}\label{Eq. Expansions}
		a_j=\sum_{n=0}^\infty a_{j,n} P^n,\quad  c=\sum_{n=0}^\infty c_n P^n,\quad A=\sum_{n=0}^\infty A_n P^n,\quad A_I=\sum_{n=0}^\infty A_{I,n}P^n,
	\end{equation} where $A_{n}\in (\mathcal{O}_b(D_r^d)\cap \Delta_{\nu_\ell(P)})^{N\times N}$, $c_n, A_{I,n}\in  (\mathcal{O}_b(D_r^d)\cap \Delta_{\nu_\ell(P)})^N$, $a_{j,n}\in \mathcal{O}_b(D_r^d)\cap \Delta_{\nu_\ell(P)}$, for all $j=1,\dots,d$, $I\in\N^N$ and $n\in\N$.

	\textbf{Step 1} (The coefficient $y_0$) First we determine the term $y_0$ in (\ref{Eq. Proof y}). Note that $\widehat{\yy}(\00)=y_0(\00)=\00$ since $F(\00,\00)=\00$ and $P(\00)=0$. When we plug $\widehat{\yy}$ into equation (\ref{Eq. Main Eq}) and equate in common powers of $P$, we find $y_0$ must be an analytic solution of \begin{equation}\label{Eq. y0}
		0=R\big(c+(\mu+A)y+\sum_{|I|\geq 2} A_{I}y^I\Big)=c_0+\mu y_0+R(A_0y_0)+\sum_{|I|\geq 2} R(A_{I,0}y_0^I),
	\end{equation} satisfying $J(y_0)\in \Delta_{\nu_\ell(P)}^N$. Note that in the last equality we have used (\ref{Eq. RRR}) joint with $R(c)=c_0$, $R(A)=A_0$ and $R(A_I)=A_{I,0}$.
	
	We will prove that (\ref{Eq. y0}) has a unique solution in $\mathcal{O}_b(D_r^d)$ if $r>0$ is taken small enough. In order to proceed, we write (\ref{Eq. y0}) as the fixed point equation \[y=G(y),\quad G(y):=-\mu^{-1}R\Big(c_0+A_0 y+\sum_{|I|\geq 2} A_{I,0}y^I\Big).\] By reducing $r>0$ we show there is $\epsilon>0$ such that $G:\overline{B}_{\epsilon}\to \overline{B}_{\epsilon}$ is well-defined and a contraction, where $\overline{B}_{\epsilon}:=\{y\in\mathcal{O}_b(D_r^d)^N : \|y\|_0\leq \epsilon, y(\00)=\00 \}$ which is closed. Then, by Banach's fixed point theorem, $G$ has a unique fixed point. 
	
	Let us check first that $G$ maps $\overline{B}_{1/2\delta}$ to $\mathcal{O}_b(D_r^d)^N$, where $\delta$ is as in (\ref{Eq. Bound F K d}): by using (\ref{Eq. Q R proof}) for $R$ and (\ref{Eq. Bound F K d}) we see that \begin{align*}
		\|G(y)\|_0&\leq \|\mu^{-1}\|_0 \|R\|\Big(\|c_0\|_0+\|A_0\|_0\|y\|_0+\sum_{|I|\geq 2} \|A_{I,0}\|_0\|y\|_0^{|I|} \Big) \\
		&\leq \|\mu^{-1}\|_0 \|R\|^2\Big(\|c\|_0+\|A\|_0\|y\|_0+\sum_{|I|\geq 2} K\delta^{|I|}\|y\|_0^{|I|} \Big) .
	\end{align*} But the identity $\sum_{|I|\geq 1} \a^{|I|}=(1-\a)^{-N}-1$, $|\a|<1$, shows that \[\|G(y)\|_0\leq \|\mu^{-1}\|_0 \|R\|^2\Big(\|c\|_0+\|A\|_0\|y\|_0+  K\delta\left(2^N-1\right)\|y\|_0 \Big) ,\] thus, $\|G(y)\|_0$ is finite as desired. Now, if $y+h, y\in \overline{B}_{1/2\delta}$ we also have \[\|G(y+h)-G(y)\|_0\leq \|\mu^{-1}\|_0 \|R\|^2\Big(\|A\|_0\|h\|_0+\sum_{|I|\geq 2} K\delta^{|I|}\|(y+h)^I-y^I\|_0\Big).\] Taking into account the inequality \[\|(y+h)^I-y^I\|_0\leq |I| (\|y\|_0+\|h\|_0)^{|I|-1} \|h\|_0,\] that follows readily by induction on $|I|$, we obtain $\|G(y+h)-G(y)\|_0$ is bounded by \begin{align*}
		& \|\mu^{-1}\|_0 \|R\|^2\Big(\|A\|_0+K \sum_{|I|\geq 2} |I|\delta^{|I|} (\|y\|_0+\|h\|_0)^{|I|-1} \Big)\|h\|_0\\
		=& \|\mu^{-1}\|_0 \|R\|^2\Big(\|A\|_0+K\delta g\big(\delta(\|y\|_0+\|h\|_0)\big) \Big)\|h\|_0,
	\end{align*} where $g(\a)=\sum_{|I|\geq 2} |I|\a^{|I|-1}=\frac{d}{d\a} \Big( \sum_{|I|\geq 2} \a^{|I|}\Big)=N((1-\a)^{-N-1}-1)$, for $|\a|<1$. Since $g$ is continuous and $g(0)=0$, we can choose  $0<\epsilon<\min\{1,1/2\delta\}$ such that $\|\mu^{-1}\|_0 \|R\|^2K\delta \cdot g(\epsilon)<1/4$. Also, since $A(\00)=\00$, we can reduce $r$ to assure that $\|\mu^{-1}\|_0 \|R\|^2\|A\|_0<1/4$. Therefore, if $\|y\|_0+\|h\|_0\leq \epsilon$, we conclude \[\|G(y+h)-G(y)\|_0\leq \frac{1}{2}\|h\|_0.\] But $c(\00)=F(\00,\00)=\00$, and since $\epsilon$ has been fixed, we can reduce $r$ again to have $\|\mu^{-1}\|_0 \|R\|^2\|c\|_0<\epsilon/2$. By applying the previous inequality to $y=0$ we find  $\|G(h)\|_0\leq \|G(h)-G(0)\|_0+\|G(0)\|_0\leq \frac{1}{2}\|h\|_0+\epsilon/2$. Thus, $G:\overline{B}_{\epsilon}\to \overline{B}_{\epsilon}$ has the desired properties.
	
	Several remarks are at hand. First, if $y_0$ is the solution of equation (\ref{Eq. y0}), $J(y_0)$ will be the unique formal solution of (\ref{Eq. y0}) and it is convergent. But $R(y_0)$ is another analytic solution of (\ref{Eq. y0}), thus $J(y_0)=J(R(y_0))\in \Delta_{\nu_\ell(P)}^N$. Second, there is a direct way to find a solution of (\ref{Eq. y0}) as follows: we find first a solution $Y_0(\xx)$ of $F(\xx,\yy(\xx))=\00$, with the aid of the holomorphic implicit function theorem -it can be applied since $F(\00,\00)=\00$  and $\frac{\d F}{\d \yy}(\00,\00)=\mu$ is invertible-. Then, it follows by applying $R$ to the previous equation that $y_0=R(Y_0)$ is the solution of (\ref{Eq. y0}), since we already know it is unique.

	\textbf{Step 2} (Recurrence equations for $y_n$) We can now assume $y_0=0$ by making the change of variables $y\mapsto y-y_0$ in the initial equation (\ref{Eq. Main Eq}). In fact, after doing so, we obtain a similar PDE such that $P$ divides $c$ and we search for a formal solution $\widehat{\yy}=\sum_{n=1}^\infty y_n P^n$ which is divisible by $P$.
	
	To find the recurrence equations satisfied by the $y_n$ we start with the right-hand side of (\ref{Eq. Main Eq}). By using the identity (\ref{Eq. Product}) we find 
	\[A\cdot \widehat{\yy}=\sum_{n=1}^\infty \left(\sum_{k=1}^n \sum_{j=1}^k RQ^{n-k}(A_{k-j}y_j)\right)P^n.\] For the non-linear term, analogously to (\ref{Eq. ProductS}), we have the decomposition \begin{align*}
		\sum_{|I|\geq 2} A_I \widehat{\yy}^I&=\sum_{k=2}^\infty\, \sum_{\ast_k} A_{I,m} \prod_{{1\leq l\leq N}\atop {1\leq j\leq i_l}} y_{l,n_{l,j}}\,P^k\\
		&=\sum_{k=2}^\infty\sum_{p=0}^\infty \, \sum_{\ast_k} RQ^p\bigg(A_{I,m} \prod_{{1\leq l\leq N}\atop {1\leq j\leq i_l}} y_{l,n_{l,j}}\bigg)\,P^{k+p}\\
		&=\sum_{n=2}^\infty\!\left[\sum_{k=2}^n \, \sum_{\ast_k} RQ^{n-k}\bigg(\!A_{I,m} \prod_{{1\leq l\leq N}\atop {1\leq j\leq i_l}} y_{l,n_{l,j}}\bigg)\!\right]P^{n}.
	\end{align*} where the sum $\sum_{\ast_k}$ is taken over all $I\in\N^N$  such that $2\leq |I|\leq k$, $m$ satisfying $0\leq m\leq k-|I|$, and $n_{l,j}\geq 1$ such that  $k=m+n_{1,1}+\cdots+n_{1,i_1}+\cdots+n_{N,1}+\cdots+n_{N,i_N}$. Note in particular that $n_{l,j}<k\leq n$ and thus no component of $y_n=(y_{n,1},\dots,y_{n,d})$ appears in the coefficient corresponding to $P^n$.
	
	For the left-hand side of (\ref{Eq. Main Eq}), using the hypothesis $L(P)=P\cdot h$, for some $h\in\mathcal{O}_b(D_r^d)$, we can write \begin{equation}\label{Eq. P L y} P\cdot L(\widehat{\yy})=\sum_{n=1}^\infty \left(L(y_{n-1}) +ny_nL(P)\right)P^n=\sum_{n=2}^\infty \left(L(y_{n-1}) +(n-1)h y_{n-1}\right)P^n.\end{equation} 
	
	Now, equating both sides of (\ref{Eq. Main Eq}) in common power series of $P$ we obtain the recurrence \begin{align*}
		\sum_{k=2}^n &RQ^{n-k}(L(y_{k-1})) +\sum_{k=2}^n(k-1)RQ^{n-k}(h y_{k-1})=c_n\\
		&+\mu y_n+\sum_{k=1}^n \sum_{j=1}^k RQ^{n-k}(A_{k-j}y_j)+\sum_{k=2}^n \, \sum_{\ast_k} RQ^{n-k}\bigg(\!A_{I,m} \prod_{{1\leq l\leq N}\atop {1\leq j\leq i_l}} y_{l,n_{l,j}}\!\bigg).
	\end{align*} Equivalently, we have \begin{equation}\label{Eq. mu yn}
		\mu y_n+R(A_0y_n)=b_n:=e_n+\sum_{k=2}^n(k-1)RQ^{n-k}(h y_{k-1}),
	\end{equation} for all $n\geq1$, where \begin{align*}
		e_n=&-c_n+\sum_{k=2}^n RQ^{n-k}(L(y_{k-1}))-\sum_{k=1}^{n-1} \sum_{j=1}^k RQ^{n-k}(A_{k-j}y_j)\\ &-\sum_{j=1}^{n-1} R(A_{n-j}y_j)
		-\sum_{k=2}^n \, \sum_{\ast_k} RQ^{n-k}\bigg(A_{I,m} \prod_{{1\leq l\leq N}\atop {1\leq j\leq i_l}} y_{l,n_{l,j}}\bigg).\end{align*}  Note in particular that $b_1=-c_1$.

	Equation (\ref{Eq. mu yn}) can be solved as follows: consider $Y_n=(\mu+A_0)^{-1}(b_n)$, where we have if necessary, reduced $r$ to ensure $\mu+A_0(\xx)$ is invertible for all $|\xx|\leq r$. Then $R(Y_n)$ solves (\ref{Eq. mu yn}), as we see by applying $R$ to $(\mu+A_0)Y_n=b_n$ and recalling that $R(b_n)=b_n$. To check uniqueness, note that if $y_{n}$ and $w_n$ are solutions, then $R((\mu+A_0)(y_n-w_n))=0$, so $(\mu+A_0)(y_n-w_n)=h_1 P$, for some $h_1\in\mathcal{O}$. Thus, $y_n-w_n=R(y_n-w_n)=R((\mu+A_0)^{-1}h_1P)=0$. In conclusion, we can find recursively the coefficients $y_n$ by means of the formulas \begin{equation}\label{Eq. yn=R bn}
		y_n=R\left((\mu+A_0)^{-1}(b_n)\right),
	\end{equation} and equation (\ref{Eq. Main Eq}) has a unique formal power series solution.
	
	\textbf{Step 3} (Majorant series) We use the majorant series technique to show that $\widehat{\yy}$ is $P$-$1$--Gevrey by proving that $\sum_{n=1}^\infty \|y_n\|_n \tau^n$ is $1$--Gevrey in $\tau$.
	
	We have chosen $r>0$ satisfying all previous requirements in order to find $y_n$ recursively. Now, we take $0<\rho <\min\{r,1\}$ satisfying $\rho^{l_j}<r$, $j=1,\dots,d$, in order to apply the bounds (\ref{Eq. Q R proof}) for functions in $\mathcal{O}_b(D_{\rho(\ell)})^N$.
	
	Let $M=\|(\mu+A_0)^{-1}\|_0>0$. By applying the Nagumo norm $\|\cdot\|_n$ to equation (\ref{Eq. yn=R bn}) and taking into account the properties developed in Propositions \ref{Prop Nagumo norms} and \ref{Generalized Weierstrass Division} we find that \begin{align*}
		\frac{\|y_n\|_n}{M\|R\|}\leq& \|c_n\|_n+\sum_{k=2}^n \|R\| \|Q\|^{n-k}(\|L(y_{k-1})\|_n+(k-1)\|h y_{k-1}\|_n)\\
		&+\sum_{k=1}^{n-1} \|R\| \|Q\|^{n-k}\sum_{j=1}^k \|A_{k-j}\|_{k-j} \|y_j\|_j+\sum_{j=1}^{n-1} \|R\| \|A_{n-j}\|_{n-j} \|y_j\|_j\\
		&+\sum_{k=2}^n \|R\| \|Q\|^{n-k}\, \sum_{\ast_k} \|A_{I,m}\|_m \prod_{{1\leq l\leq N}\atop {1\leq j\leq i_l}} \|y_{l,n_{l,j}}\|_{n_{l,j}}.
	\end{align*} To bound  the term $\|L(y_{k-1})\|_n$ use Proposition \ref{Prop Nagumo norms} (ii) and that $\rho<1$ to get \begin{align*}
		\left\|a_j\frac{\d y_{k-1}}{\d x_j}\right\|_n\leq \|a_j\|_{n-k} \left\|\frac{\d y_{k-1}}{\d x_j}\right\|_k&\leq ek\prod_{i\neq j}(\rho^{\ell_i}/2)\|a_j\|_{n-k} \|y_{k-1}\|_{k-1}\\
		&\leq ek\|a_j\|_{n-k} \|y_{k-1}\|_{k-1}.
	\end{align*} But inequality (\ref{Eq fm f0}) implies that $\|a_j\|_{n-k}\leq \|a_j\|_0.$ If $a=\|a_1\|_0+\cdots+\|a_d\|_0$, by hypothesis $a>0$, and \[\|L(y_{k-1})\|_n\leq eak\|y_{k-1}\|_{k-1}.\] On the other hand, $\|h y_{k-1}\|_n\leq \|h\|_{n-k+1} \|y_{k-1}\|_{k-1}\leq \|h\|_0\|y_{k-1}\|_{k-1}$. Thus, we find that \begin{align}
		\label{Eq. yn MR} \frac{\|y_n\|_n}{M\|R\|}\leq& \|c_n\|_n+\|R\|(ea+\|h\|_0)\sum_{k=2}^n  k\|Q\|^{n-k}\|y_{k-1}\|_{k-1}\\
		\nonumber +&\|R\|\sum_{k=1}^{n-1} \|Q\|^{n-k} \sum_{j=1}^k  \|A_{k-j}\|_{k-j} \|y_j\|_j+\|R\|\sum_{j=1}^{n-1} \|A_{n-j}\|_{n-j} \|y_j\|_j\\
		\nonumber +&\|R\|\sum_{k=2}^n  \|Q\|^{n-k}\, \sum_{\ast_k} \|A_{I,m}\|_m \prod_{{1\leq l\leq N}\atop {1\leq j\leq i_l}} \|y_{n_{l,j}}\|_{n_{l,j}}.
	\end{align} If we divide by $n!$ and using that $m!k!\leq (m+k)!$ we conclude that \begin{align*}
		&\frac{\|y_n\|_n}{M\|R\| n!}\leq \frac{\|c_n\|_n}{n!}+\|R\|(ea+\|h\|_0)\sum_{k=2}^n  \frac{\|Q\|^{n-k}}{(n-k)!}\, \frac{\|y_{k-1}\|_{k-1}}{(k-1)!}\\
		&+\|R\|\sum_{k=1}^{n-1} \frac{\|Q\|^{n-k}}{(n-k)!} \sum_{j=1}^k  \frac{\|A_{k-j}\|_{k-j}}{(k-j)!} \frac{\|y_j\|_j}{j!}+\|R\|\sum_{j=1}^{n-1} \frac{\|A_{n-j}\|_{n-j}}{(n-j)!} \frac{\|y_j\|_j}{j!}\\
		&+\|R\|\sum_{k=2}^n  \frac{\|Q\|^{n-k}}{(n-k)!}\, \sum_{\ast_k} \frac{\|A_{I,m}\|_m}{m!} \prod_{{1\leq l\leq N}\atop {1\leq j\leq i_l}} \frac{\|y_{n_{l,j}}\|_{n_{l,j}}}{n_{l,j}!}.
	\end{align*}
	
	\noindent Let us define the sequence $z_n$ recursively by \begin{align}
		\label{Eq. zn recurrence} \frac{z_n}{M\|R\|}=&\frac{\|c_n\|_n}{n!}+\|R\|(ea+\|h\|_0)\sum_{k=2}^n  \frac{\|Q\|^{n-k}}{(n-k)!}\, z_{k-1}\\\nonumber &+\|R\|\sum_{k=1}^{n-1} \frac{\|Q\|^{n-k}}{(n-k)!} \sum_{j=1}^k  \frac{\|A_{k-j}\|_{k-j}}{(k-j)!} z_j
		+\|R\|\sum_{j=1}^{n-1} \frac{\|A_{n-j}\|_{n-j}}{(n-j)!} z_j\\
		\nonumber &+\|R\|\sum_{k=2}^n  \frac{\|Q\|^{n-k}}{(n-k)!}\, \sum_{\ast_k} \frac{\|A_{I,m}\|_m}{m!} \prod_{{1\leq l\leq N}\atop {1\leq j\leq i_l}} z_{n_{l,j}},
	\end{align} where $z_1=M\|R\|\|c_1\|_1$. Since the terms of the previous equation are all nonnegative real numbers, we find inductively that \begin{equation}\label{Eq. yn zn}\frac{\|y_n\|_n}{n!}\leq z_n.\end{equation} On the other hand, we consider the generating power series \begin{equation}\label{Eq. Generating series}
		\overline{c}=\sum_{n=1}^\infty \frac{\|c_n\|_n}{n!}\tau^n 
		,\quad  \overline{A}=\sum_{n=0}^\infty \frac{\|A_n\|_n}{n!}\tau^n,\quad \overline{F}=\sum_{m\geq 0, |I|\geq 2} \frac{\|A_{I,m}\|_m}{m!}\tau^m Z^{|I|},\end{equation}  which are convergent. In fact, recalling the expansions in (\ref{Eq. Expansions}), Corollary \ref{Coro Decomposition convergent} and the inequalities in (\ref{Eq. Q R proof}) and (\ref{Eq fm f0}) --recall $\rho<1$--, we have \begin{equation}\label{Eq. Convergence c}\|c_n\|_n=\|R\circ Q^n(c)\|_n\leq \|R\|\|Q\|^n \|c\|_n\leq \|R\|\|Q\|^n \|c\|_0.\end{equation} thus $\overline{c}$ is actually entire. The same argument applies for $\overline{A}$. Now, for $\overline{F}$ we use inequality (\ref{Eq. Bound F K d})
	to write \[\|A_{I,m}\|_m=\|R\circ Q^m(A_I)\|_m\leq \|R\|\|Q\|^m \|A_I\|_0\leq \|R\|\|Q\|^m K\delta^{|I|},\] and thus \begin{align}
		\label{Eq. Convergence F} \sum_{|I|=j} \|A_{I,m}\|_m &\leq K\|R\| \|Q\|^m  \sum_{|I|=j}  \delta^{|I|} \\ \nonumber
		&=K\|R\| \|Q\|^m  \binom{j+N-1}{N-1} \delta^{j}\leq (K\|R\|2^{N-1}) \|Q\|^m (2\delta)^{j},
	\end{align} since the number of solutions $I\in\N^N$ of $|I|=j$ is $\binom{j+N-1}{N-1}$ which is less than $2^{j+N-1}$. Therefore, the coefficient in $\tau^m Z^j$ of $\overline{F}$ is bounded by $(K\|R\|2^{N-1})  (2\delta)^{j} \|Q\|^m/m!$ proving the convergence of $\overline{F}$.	
	
	Using these series and equation (\ref{Eq. zn recurrence}), we find $Z(\tau)=\sum_{n=1}^\infty z_n \tau^n$ is a formal solution of the analytic equation \[E(\tau,Z(\tau))=0,\] where
	\begin{align*}
		E(\tau,Z):=&-\frac{Z}{M\|R\|}+\overline{c}(\tau)+\|R\|(ea+\|h\|_0)  e^{\|Q\|\tau} \tau Z\\&+\|R\| (e^{\|Q\|\tau}\overline{A}(\tau)-\|A_0\|_0)Z+\|R\|e^{\|Q\|\tau}\overline{F}(\tau,Z).
	\end{align*} But (\ref{Eq. Generating series}) shows that $\overline{c}(0)=\overline{F}(0,0)=\frac{\d \overline{F}}{\d Z}(0,0)=0$ and $\overline{A}(0)=\|A_0\|_0$. Therefore, 
	$E(0,0)=\overline{c}(0)+\|R\|\overline{F}(0,0)=0$ and $\frac{\d E}{\d Z}(0,0)=-{1}/{M\|R\|}+\|R\|(\overline{A}(0)-\|A_0\|_0)+\|R\|\frac{\d \overline{F}}{\d Z}(0,0)=-{1}/{M\|R\|}\neq 0$ 	
	Then the holomorphic implicit function theorem implies this equation has a unique convergent power series solution at the origin, thus it must be $Z(\tau)$, so it is convergent. By (\ref{Eq. yn zn}), $\sum_{n=1}^\infty \|y_n\|_n \tau^n$ is $1$--Gevrey in $\tau$ as desired.
\end{proof}

\begin{proof}[Proof of Theorem \ref{Thm 2}] Regarding the previous proof, only some minor changes are required so establish the result. While Step 0 and Step 1 remain the same, in Step 2 the recurrence for $y_n$ takes the form \begin{equation}\label{Eq. mu yn 2}
		\mu y_n-nR(L(P)y_n)+R(A_0y_n)=d_n:=e_n+\sum_{k=1}^{n-1} kRQ^{n-k}(L(P)y_{k})
		,\end{equation} for all $n\geq1$, where $e_n$ is as before. Then $d_n$ and $b_n$ differ only in the previous sum, that we bound by shifting one index, as follows \begin{equation}\label{Eq. last one}
		\left\|\sum_{k=2}^{n} (k-1)RQ^{n-k+1}(L(P)y_{k-1})\right\|_n\leq C \sum_{k=2}^{n}  k \|Q\|^{n-k}  \|y_{k-1}\|_{k-1},
	\end{equation} where $C=\|R\|\|Q\|\|L(P)\|_0$. In the current case, the solution of (\ref{Eq. mu yn 2}) is  \[y_n=R\left((\mu-nL(P)I_N+A_0)^{-1}(d_n)\right),\] where $I_N$ is the identity matrix of size $N$. To make this formula meaningful, it is enough to prove $\mu-nL(P)(\xx)I_N+A_0(\xx)$ is invertible for all $n\geq 1$ and $|\xx|\leq r$, if $r$ is sufficiently small. To proceed let us recall that if $B\in\C^{N\times N}$ is such that $|B|<1$ for a matrix norm $|\cdot |$, then $I_N-B$ is invertible, $(I_N-B)^{-1}=\sum_{n=0}^\infty B^n$, and $|(I_N-B)^{-1}|\leq (1-|B|)^{-1}$. Here as before we use $|B|=\max_{1\leq i\leq N} \sum_{j=1}^N |B_{i,j}|$. Now, since $L(P)(\00)\neq 0$, we can choose a small $r>0$ such that $\a=\inf_{|x|\leq r}|L(P)(\xx)|>0$. Thus, if $n>\|\mu+A_0\|_0/\a$, we see that $\left|\mu+A_0(\xx)\right|/|nL(P)(\xx)|\leq \|\mu+A_0\|_0/n\a<1$, for all $|\xx|\leq r$, so $\mu-nL(P)+A_0$ is invertible and \begin{align*}
		|(\mu-nL(P)(\xx)+A_0(\xx))^{-1}|		&=\frac{1}{n|L(P)(\xx)|}\left|\left(I_N-\frac{1}{nL(P)(\xx)}(\mu+A_0(\xx))\right)^{-1}\right|\\
		&\leq \frac{1/\a n}{1-\frac{\|\mu+A_0\|_0}{\a n}}=\frac{1}{\a n-\|\mu+A_0\|_0}.
	\end{align*} For $n\leq \|\mu+A_0\|_0/\a$, by hypothesis the remaining finite number of matrices $\mu-nL(P)(\xx)+A_0(\xx)$ are invertible at the origin. Thus, we can shrink $r$ and assume they are invertible for all $|\xx|\leq r$. In conclusion, all these matrices are invertible, and we can find $M>0$ such that \begin{equation}\label{Eq. mu Mn}
		\|(\mu-nL(P)+A_0)^{-1}\|_0\leq M/n,\quad \text{ for all } n\geq 1.
	\end{equation}
	
	At this stage we can proceed with Step 3 by using the Nagumo norms and taking into account (\ref{Eq. last one}). However, the factor $M/n$ in (\ref{Eq. mu Mn}) improves our bounds and shows that 
	\begin{align*}
		\frac{\|y_n\|_n}{M\|R\|}\leq& \|c_n\|_n+\|R\|(ea+\|Q\|\|L(P)\|_0)\sum_{k=2}^n  \|Q\|^{n-k}\,\|y_{k-1}\|_{k-1}+\cdots,
	\end{align*} where the dots indicate the remaining terms are the same as in (\ref{Eq. yn MR}). In this case it is not necessary to divide by $n!$, just by defining $z_n$ accordingly we find $\|y_n\|_n\leq z_n$, for all $n\geq 1$, and $Z(\tau)$ satisfies the analytic equation \[\widetilde{E}(\tau,Z(\tau))=0,\] where 
	\begin{align*}
		\widetilde{E}(z,Z):= &-\frac{Z}{M\|R\|}+ \widetilde{c}(\tau)+\|R\|(ea+\|Q\|\|L(P)\|_0)  \frac{\tau Z}{1-\|Q\|\tau}\\ &+\|R\| \left(\frac{\widetilde{A}(\tau)}{1-\|Q\|\tau}-\|A_0\|_0\right)Z+\|R\|\frac{\widetilde{F}(\tau,Z)}{1-\|Q\|\tau},
	\end{align*} with coefficients  $$\widetilde{c}(\tau)=\sum_{n=1}^\infty \|c_n\|_n\tau^n,\,\, 
	\widetilde{A}(\tau)=\sum_{n=0}^\infty \|A_n\|_n \tau^n,\,\, \widetilde{F}(\tau,Y)=\sum_{m\geq 0, |I|\geq 2} \|A_{I,m}\|_m \tau^mY^{|I|},$$ which are again convergent as justified by the inequalities (\ref{Eq. Convergence c}) and (\ref{Eq. Convergence F}). Since $\widetilde{E}(0,0)=\widetilde{c}(0)+\|R\| \widetilde{F}(0,0)=0$ and $\frac{\d \widetilde{E}}{\d Z}(0,0)=-1/M\|R\|+\|R\|(\widetilde{A}(0)-\|A_0\|_0)+\|R\|\frac{\d \widetilde{F}}{\d Z}(0,0)=-1/M\|R\|\neq 0$, the holomorphic implicit function theorem proves that $Z(\tau)$ and therefore $\sum_{n=1}^\infty \|y_n\|_n \tau^n$ are convergent. This proves the convergence of $\widehat{\yy}=\sum_{n=1}^\infty y_n P^n$ as required.
\end{proof}

\section{A simple extension to higher order systems}\label{Sec. Higher order}

There is a straightforward way to extend our theorems for systems of PDEs of higher order by augmenting the size of the given equation.

\begin{coro}\label{Coro. 1}
	Let $P$, $L$ and $F$ be as in Theorem \ref{Thm Main Result}, fix $u_1,\dots,u_{k-1}\in\mathcal{O}$, and consider the  system of PDEs \begin{equation}\label{Eq. main 3}
		(P\cdot L)^k(\yy)(\xx)+u_{k-1}(\xx) (P\cdot L)^{k-1}(\yy)(\xx)+\cdots+u_1(\xx) (P\cdot L)(\yy)(\xx)=F(\xx,\yy).\end{equation} Then, the following statements hold:

\begin{enumerate}[label=(\roman*)]
		
\item If $P$ divides $L(P)$, (\ref{Eq. main 3}) has a unique formal power series solution which is $P$-$1$--Gevrey.
		
\item If $L(P)(\00)\neq 0$, and $\sigma-nL(P)(\00)\neq 0$, for all $n\in\N$ and all solutions $\sigma$ of the polynomial equation \begin{equation}\label{Eq. Eigenvalues}
p_\mu\left(\sigma^{k}+u_{k-1}(\00)\sigma^{k-1}+\cdots+u_2(\00)\sigma^2+u_1(\00)\sigma\right)=0,\end{equation} where $p_\mu$ is the characteristic polynomial of $\mu$, then (\ref{Eq. main 3}) has a unique convergent power series solution.
\end{enumerate}
\end{coro} 

\begin{proof} In the variable $\boldsymbol{w}=(\boldsymbol{w}_0,\boldsymbol{w}_1,\dots,\boldsymbol{w}_{k-1})\in\C^{Nk}$, where $\boldsymbol{w}_0=\yy$, $\boldsymbol{w}_1=(P\cdot L)(\boldsymbol{w}_0), \boldsymbol{w}_2= (P\cdot L)(\boldsymbol{w}_1),\dots, \boldsymbol{w}_{k-1}=(P\cdot L)(\boldsymbol{w}_{k-2})$, (\ref{Eq. main 3}) can be written as \begin{align*}
		P\cdot L(\boldsymbol{w})&=G(\xx,\boldsymbol{w})\\
		&=\left(\boldsymbol{w}_1,\boldsymbol{w}_2\dots,\boldsymbol{w}_{k-1},F(\xx,\boldsymbol{w}_0)-u_1(\xx)\boldsymbol{w}_1-\cdots-u_{k-1}(\xx) \boldsymbol{w}_{k-1}\right),
	\end{align*} which has the form of equation (\ref{Eq. Main Eq}). Then, the results follow from Theorem \ref{Thm Main Result} and \ref{Thm 2} by noticing that
	\[\frac{\d G}{\d \boldsymbol{w}}(\00,\00)=\left(\begin{array}{ccccc}
		0 & I_N & 0 & \cdots & 0\\
		0 & 0 & I_N & \cdots & 0\\
		\vdots & \vdots & \vdots & \ddots & \vdots\\
		0 & 0 & 0 & \cdots & I_N\\
		\mu & -u_1(\00)I_N & -u_2(\00)I_N & \cdots & -u_{k-1}(\00)I_N\\
	\end{array}\right)\in\C^{Nk}\times\C^{Nk}\] is invertible with eigenvalues given by the solutions of (\ref{Eq. Eigenvalues}), see, e.g., Refs. \cite[p. 293]{Barnett1981}, \cite{CDet}.
\end{proof}

\section{Examples}\label{Sec. Examples}

Theorem \ref{Thm Main Result} has a general nature and recovers many examples of Gevrey formal power series solutions of ODEs and PDEs that have been treated in the literature. We conclude this paper explaining some of them.

\begin{eje} The solution of (\ref{Eq. Main Eq}) is generically divergent, but there are cases where it can be convergent. This is evidenced already in the case of one variable: while Euler's equation $x^2y'+y=x$ has the $x$-$1$--Gevrey solution $\widehat{y}(x)=\sum_{n=0}^\infty (-1)^n n!x^{n+1}$, the equation $x^2y'+y=x+x^2$ has $\widehat{y}(x)=x$ as analytic solution. More examples can be obtained by taking $f\in\C\{z\}$ and $P\in\mathcal{O}$, $P(\00)=0$. If $L(P)=0$, then the solution of \[P(\xx)L(y)=y-f(P(\xx)),\quad \text{ is }\quad  \widehat{y}(\xx)=f(P(\xx)),\] which is convergent.
\end{eje}

\begin{eje}\label{Ex. ODE} We consider the equation \[x_1x_2\frac{\d y}{\d x_1}=\mu y-\frac{x_1}{1-x_1},\] where $\mu\neq 0$ is constant. A way to find its unique formal power series solution is to plug $\widehat{y}=\sum_{n=0}^\infty y_n(x_2)x_1^n$ into the equation and then equate common powers of $x_1$. Thus, we find $y_0(x_2)=0$, $y_n(x_2)=(\mu-n x_2)^{-1}$, $n\geq 1$, and the formal solution is equal to \[\widehat{y}(x_1,x_2)=\sum_{n\geq1, m\geq 0} \frac{n^m}{\mu^{m+1}} x_1^n x_2^m.\] We see $\widehat{y}$ is $x_2$-$1$--Gevrey by direct inspection or by applying Theorem \ref{Thm Main Result} to $P=x_2$ and $L=x_1\frac{\d}{\d x_1}$ since $L(P)=0$. However, $\widehat{y}$ is not $x_1$-$1$--Gevrey, i.e., $P=x_1$, $L=x_2\frac{\d}{\d x_1}$ is not a valid choice: as a power series in $x_1$, $y_n(x_2)$ is analytic on the disc $\{x_2\in\C : |x_2|<|\mu|/n\}$, so there is no common neighborhood of the origin where all the $y_n(x_2)$ are defined.
\end{eje}

\begin{eje} Equation (\ref{Eq. Main Eq}) includes the case of singularly perturbed and doubly singular ODEs \begin{equation}\label{Eq Example 1}
		Q(\ee)x^{k+1}\frac{\d\yy}{\d x}(x,\ee)=F(x,\ee,\yy),
	\end{equation} where $x\in(\C,0)$, $\ee=(\varepsilon_1,\dots,\varepsilon_m)\in(\C^m,\00)$, $Q$ is analytic at the origin and  $k\geq -1$ is an integer. In the regular case $k=-1$ or $0$ and $Q(\00)\neq 0$, if there exists a formal solution, it is convergent. In the irregular case, if $Q(\00)\neq 0$, we can interpret $\ee$ as regular parameters and the classical theory establishes that the formal solution of (\ref{Eq Example 1}) is $1/k$--Gevrey in $x$, uniformly in $\ee$, see Ref.  \cite{Sibuya1}, i.e., it is a $x^k$-$1$--Gevrey series. Equation (\ref{Eq Example 1}) was studied by W. Balser and V. Kostov in Ref. \cite{BalserKostov02} for $m=1$, $k=0$, and by W. Balser and J. Mozo-Fern\'{a}ndez in Ref. \cite{BalserMozo02} for $m=1$, $k\geq 1$, both when $Q(\varepsilon)=\varepsilon$ and in the linear case $F(x,\varepsilon,y)=A(x,\varepsilon)y-f(x,\varepsilon)$, proving the summability of the formal solution in the perturbation parameter $\varepsilon$, in adequate domains of $x$. On the other hand, M. Canalis-Durand, J.P. Ramis, R. Sch\"{a}fke and Y. Sibuya in Ref. \cite{CDRSS} studied this equation  when $m=1$, $k=-1$ and $Q(\varepsilon)=\varepsilon^\sigma$, $\sigma\geq 1$ a positive integer. In particular, they showed that the solution is $1/\sigma$--Gevrey in $\varepsilon$, uniformly in $x$. Later on, M. Canalis-Durand, J. Mozo-Fern\'{a}ndez and R. Sch\"{a}fke in Ref. \cite{CDMS} considered the case $m=1$, $Q(\varepsilon)=\varepsilon^q$, and $k,q\geq 1$, and they proved the $\varepsilon^q x^k$-$1$--summability of the formal power series solution and the singular directions are determined by the solutions of $\det\left(k\eta^q\xi^kI_N-\mu\right)=0,$ in the two-dimensional $(\xi,\eta)-$Borel space. We can recover all these divergence rates using Theorem \ref{Thm Main Result}:
	
	\begin{enumerate}[label=(\roman*)]	
		\item If $k=-1$ and $Q(\00)=0$, by choosing $P(x,\ee)=Q(\ee)$ and $L=\d_x$, we have $L(P)=0$. Thus, the solution is $Q(\ee)$-$1$--Gevrey.
		
		\item If $k\geq 0$, we take $P(x,\ee)=x^kQ(\ee)$ and $L=x\d_x$, since $L(P)=kx^kQ(\ee)=kP$. Thus, the solution is $x^kQ(\ee)$-$1$--Gevrey. If $Q(\00)\neq 0$, this means the solution is a $x^k$-$1$--Gevrey series.
	\end{enumerate} 
\end{eje}

\begin{eje}	Let $\ee$ and $Q$ be as in the previous example and assume $P(\00)=0$. If $L=a_1(\xx,\ee)\d_{x_1}+\cdots+a_d(\xx,\ee)\d_{x_d}$, the system of PDEs \begin{equation}\label{Eq. QP}
		Q(\ee)P(\xx)L(\yy)(\xx,\ee)=F(\xx,\ee,\yy),
	\end{equation} can be seen as a singularly perturbed problem where the perturbation is given by $Q$ when $Q(\00)=0$. If $P$ divides $L(P)$, then $L(Q P)=Q L(P)$ is divisible by $QP$ and we can apply Theorem \ref{Thm Main Result} to conclude the system has a unique formal power series solution which is $Q(\ee)P(\xx)$-$1$--Gevrey. Let us consider several instances of this situation. First, assume $P\in\C[\xx]$ is a quasi--homogeneous polynomial, i.e., there are rational numbers $\lambda,\lambda_1,\dots,\lambda_d>0$ such that $P(t^{\lambda_1}x_1,\dots,t^{\lambda_d}x_d)=t^\lambda P(\xx)$. Then \begin{equation*}\label{Eq. L lambda} L_{\boldsymbol{\lambda}}:=\lambda_1 x_1 \d_{x_1}+\cdots+\lambda_d x_d \d_{x_d},
	\end{equation*} satisfies $L_{\boldsymbol{\lambda}}(P)=\lambda P$, and the solution of (\ref{Eq. QP}) is a $Q(\ee)P(\xx)$-$1$--Gevrey. 
	Second, consider the choice \[P(\xx)=\xx^\aa \text{ and } L=\sum_{j=1}^d b_j(\xx) x_j\d_{x_j},\] with $b_j$ holomorphic near the origin. Then $L(\xx^\aa)=\xx^\aa \sum_{j=1}^d \a_j b_j(\xx)$, thus the solution of (\ref{Eq. QP}) is $Q(\ee)\xx^\aa$-$1$--Gevrey. Third, families of PDEs with normal crossings given by  \begin{equation}\label{Eq. eaxaL}
		\ee^{\aa'}\xx^{\aa}L_{\boldsymbol{\lambda}}(\yy)(\xx,\ee)={F}(\boldsymbol{x},\boldsymbol{\varepsilon},\boldsymbol{y}),\end{equation}  $L_{\boldsymbol{\lambda}}$ and $\aa$ as before, $\aa'\in\N^m$, and $\boldsymbol{\lambda}=(\lambda_1,\dots,\lambda_d)\in(\C^\ast)^d$. Since \[L_{\boldsymbol{\lambda}}(\ee^{\aa'}\xx^{\aa})=\left<\boldsymbol{\lambda},\aa\right>\ee^{\aa'}\xx^{\aa},\] where $\left<\boldsymbol{\lambda},\aa\right>:=\lambda_1\a_1+\cdots+\lambda_d\a_d$, we obtain a $\ee^{\aa'}\xx^{\aa}$-$1$--Gevrey solution. We remark these equations have been studied by H. Yamazawa and M. Yoshino in Ref. \cite{YamazawaYoshino15} in the case $m=1$, $\aa=\00$, $\mu=\text{diag}(\mu_1,\dots,\mu_d)$ a diagonal matrix and $\lambda_j, \text{Re}(\mu_k)>0$, for all $j,k=1,\dots,d$. In fact, the authors proved the $1$--summability in $\varepsilon=\eta$ of the formal solution, uniformly in $\xx$. In this trend, and assuming that $\boldsymbol{\lambda}$ has, up to a non-zero constant, positive entries, J. Mozo-Fern\'{a}ndez and the first author in Ref. \cite{CM2} studied these equations for the case $d=2$ and $m=0$ proving the solution is actually $x_1^{\a_1}x_2^{\a_2}$-$1$--summable. Later on, this was generalized by the first author in Ref. \cite{Carr1} for any $d\geq 2$ and $m$ by using an adapted Borel--Laplace method: the formal solution is $\ee^{\aa'}\xx^\aa$-$1$--summable and the singular directions are determined by the solutions of $\det(\left<\ll,\aa\right>\xxi^{\aa}\boldsymbol{\eta}^{\aa'}I_N-\mu)=0,$ in the $(d+m)$-dimensional $(\xxi,\boldsymbol{\eta})-$space. 
	
	Finally, another instance of equation (\ref{Eq. QP}) is the family of scalar singular first-order linear PDEs of nilpotent type \[(\a(x)+\beta(x,y))y\d_xu+(a+b(x,y))y^2 \d_y u +(1+\mathrm{a}(x,y)y)u=f(x,y),\] where $\a(0)\neq 0$ and $\beta(x,0)\equiv b(x,0)\equiv 0$. We obtain a unique  $y$-$1$--Gevrey series solution by taking $P(x,y)=y$ and $L=(\a+\b)\d_x+(a+b)y\d_y$. These equations were studied by M. Hibino in Refs. \cite{Hibino2006I,Hibino2006II}  proving the $1$--summability in $y$, uniformly in $x$, under conditions on $\a$, and on the analytic continuation and exponential growth of  $\beta,b, \mathrm{a}$ and $f$.
\end{eje}

\begin{eje} We examine a particular case of equation (\ref{Eq. eaxaL}) without the singular parameter $\ee$ but this time decomposing the series involved as a power series with coefficients series in a monomial. Taking $\aa\in \N^d\setminus\{\00\}$, $\mu\in\C^\ast$, $\boldsymbol{\lambda}=(\lambda_1,\dots,\lambda_d)\in (\C^\ast)^d$, and $c(\xx)=\sum_{\bb\in\N^d} a_\bb \xx^\bb\in\C\{\xx\}$, consider \[\xx^\aa L_{\boldsymbol{\lambda}}(\yy)=\xx^\aa \left( \lambda_1 x_1 \d_{x_1}\yy +\cdots+ \lambda_d x_d \d_{x_d}\yy \right)=\mu\yy-c(\xx).\] This equation has generically a $\xx^\aa$-$1$--Gevrey formal solution $\widehat{y}$. To find it we can reduce the problem to solve a family of ODEs as follows: write \begin{equation}\label{Eq. Decomposition xa}
		\widehat{y}(\xx)=\sum_{\aa\not\leq \bb} \xx^\bb \widehat{y}_\bb(\xx^\aa), \quad c(\xx)=\sum_{\aa\not\leq \bb} \xx^\bb c_\bb(\xx^\aa),
	\end{equation} as power series with coefficients series in $\xx^\aa$, according to the decomposition $c_\bb(t)=\sum_{n=0}^{\infty} a_{n\aa+\bb} t^n$. Then, plug $\widehat{y}$ into the equation and equate the common terms in $\xx^\bb$. It follows the initial problem is equivalent to solve the family of independent ODEs \[\left<\ll,\aa\right> t^2\widehat{y}_\bb'(t)=(\mu-\left<\ll,\bb\right>t)\widehat{y}_\bb(t)-c_\bb(t),\quad \aa\not\leq \bb.\] Thus each $\widehat{y}_\bb$ is uniquely determined and generically $t$-$1$--Gevrey. For instance, if $c(\xx)=\xx^{\bb}$ we have two cases: if $\left<\ll,\aa\right>=0$, the solution is $\widehat{y}(\xx)=\frac{\xx^\bb}{1-\left<\ll,\bb\right>\xx^\aa}$ which is convergent. Otherwise, after some calculations we find the $\xx^\aa$-$1$--Gevrey solution \[\widehat{y}(\xx)=\sum_{n=0}^\infty \binom{-\left<\ll,\bb\right>/\left<\ll,\aa\right>}{n}(-1)^n n! \frac{\left<\ll,\aa\right>^n}{\mu^{n+1}} \xx^{n\aa+\bb}.\] Note $\widehat{y}$ reduces to a polynomial if $\left<\ll,m\aa+\bb\right>=0$, for some $m\geq 0$. 	
	
	As an explicit example in two dimensions consider \[x_1x_2\left( x_1 \d_{x_1} y- x_2 \d_{x_2} y\right)=\mu y-(1-x_1)^{-1}(1-x_2)^{-1}.\] Then, decomposition (\ref{Eq. Decomposition xa}) takes the form \begin{align*}
		\widehat{y}(x_1,x_2)&=\sum_{n=0}^\infty y_{(n,0)}(x_1x_2) x_1^n+\sum_{n=1}^\infty y_{(0,n)}(x_1x_2) x_2^n,\\ 
		\frac{1}{(1-x_1)(1-x_2)}&=\sum_{n,m\geq 0} x_1^nx_2^m=\frac{1}{1-x_1x_2}+\sum_{n=1}^\infty \frac{x_1^n+x_2^n}{1-x_1x_2},
	\end{align*}
	and we find the coefficients are equal to  \[y_{(n,0)}(t)=\frac{1}{(1-t)(\mu-nt)},\quad y_{(0,n)}(t)=\frac{1}{(1-t)(\mu+nt)},\text{ valid for }|t|<\frac{|\mu|}{n}.\] By using the Taylor series at the origin of the previous functions we can determine $\widehat{y}$. The relation between $\widehat{y}$ and the solution 
	\[y_0(x_1,x_2)=\frac{1}{1-x_1x_2}\sum_{n=0}^\infty \frac{x_1^n}{\mu-n x_1x_2}+\frac{1}{1-x_1x_2}\sum_{n=1}^\infty \frac{x_2^n}{\mu +nx_1x_2},\] which is analytic on $\{(x_1,x_2)\in\C^2 : |x_1|,|x_2|<1, x_1x_2\neq \mu/n, n\geq 1\}$, is that $y_0$ is the $x_1x_2$-$1$--sum of $\widehat{y}$, see Ref. \cite[Example 2.1]{CM16} for more details on similar calculations with these series and their monomial summability.
\end{eje}

Theorem \ref{Thm Main Result} can also be applied in other situations after ramifications and punctual blow--ups. We illustrate this fact with two final examples.

\begin{eje}\label{Ex. Zhang} Consider the system of PDEs given by \begin{equation}\label{Eq. Singular}
		x_1^{p_1+1}c_1(\xx) \d_{x_1}\yy+\cdots+x_d^{p_d+1}c_d(\xx) \d_{x_d}\yy=\boldsymbol{F}(\xx,\yy),
	\end{equation} where $\mu=\frac{\partial \boldsymbol{F}}{\partial \yy}(\00,\00)$ is invertible, and we assume $1\leq p_1\leq p_j$, for all $j$. Then the system has a unique formal power series solution $\widehat{\yy}$ which is $(1/p_1,\dots,1/p_1)$--Gevrey. For instance, an explicit example is the multidimensional Euler's equation 
	\[x_1^2\d_{x_1}y+\cdots+x_d^2\d_{x_d}y+y=\xx^\11,\quad \widehat{\yy}(\xx)=\sum_{\bb\in\N^d} (-1)^{|\bb|}|\bb|!\xx^{\bb+\11},\] where $\11=(1,\dots,1)$.
	
	To prove the claim we use the punctual blow--up (\ref{Eq. Blow up}) to find that 
	\[z_1\d_{z_1}=x_1\d_{x_1}+\cdots+x_d\d_{x_d},\quad z_j\d_{z_j}=x_j\d_{x_j},\quad  j=2,\dots,d,\] and thus (\ref{Eq. Singular}) takes the form \[z_1^{p_1}{L}'(\uu)=z_1^{p_1+1}{c}'_1\d_{z_1}\uu +z_1^{p_1}\sum_{j=2}^d (z_1^{p_j-p_1}z_j^{p_j}{c}'_j-{c}'_1) z_j\d_{z_j}\uu=\boldsymbol{F}'(\zz,\uu),\] where ${c}'_j(\zz)=c_j(\xx)$, $\boldsymbol{F}'(\zz,\uu)=\boldsymbol{F}(\xx,\yy)$, and $\uu(\zz)=\yy(\xx)$. Since ${L}'(z_1^{p_1})=p_1 z_1^{p_1}{c}'_1(\zz)$ and  $\frac{\partial {\boldsymbol{F}'}}{\partial \zz}(\00,\00)=\mu$ is invertible, Theorem \ref{Thm Main Result} implies that $\widehat{\uu}(\zz)=\widehat{\yy}(\xx)$ is a $z_1^{p_1}$-$1$--Gevrey series. Then the proof of Proposition \ref{Prop. Ps sss Gevrey} shows $\widehat{\yy}$ is a $(1/p_1,\dots,1/p_1)$--Gevrey series as desired.

	On the other hand, it is worth remarking (\ref{Eq. Singular}) has been recently studied by Z. Luo, H. Chen and C. Zhang in Ref. \cite{Zhang19} for the case \[x_1^2 c_1(\xx)\d_{x_1}u+x_2^2 c_2(\xx) \d_{x_2}u=b(\xx) u-a(\xx),\]  where $a,b,c_1,c_2$ are analytic near $\00\in\C^2$ and $b(\00)c_1(\00)c_2(\00)\neq 0$. This scalar equation has a unique formal power series solution $\hat{u}$ which is Borel--summable in the variables $(x_1,x_2)$. In particular, $\hat{u}$ is $(1,1)$--Gevrey as we have shown.	
\end{eje}

\begin{eje}\label{Ex. Klimes} Consider the family of ODEs unfolding $k+1$ singularities  \begin{equation}\label{Eq. Unfolding} (x^{k+1}-\varepsilon)\frac{d\yy}{dx}=\mu \yy-f(x,\varepsilon,\yy),\end{equation} where $k$ is a positive integer, $\mu$ is an invertible matrix, $f$ is analytic at the origin in $\C\times\C\times\C^d$, $\frac{\d f}{\d \yy}(0,0,\00) =\00$, and $\varepsilon\in (\C,0)$ is a small parameter. These systems have been studied by M. Klime\v{s} in Ref. \cite{Klimes2016} for the case $k=1$ by using an adapted (unfolded) Borel-Laplace method to obtain parametric solutions bounded on certain ramified domains attached to both singularities $x=\pm \sqrt{\varepsilon}$, at which they possess a limit in a spiraling manner. 	
	
	As it is remarked in Ref. \cite[Section 2.4]{Klimes2016}, the system above has a unique formal power series solution $\widehat{\yy}$ which is $\left(\frac{1}{k},\frac{k+1}{k}\right)$--Gevrey in $(x,\varepsilon)$. We can prove this readily as follows: consider the ramification $\varepsilon=\eta^{k+1}$, and afterwards the punctual blow-up $x=z$, $\eta=z\zeta$. In these coordinates equation (\ref{Eq. Unfolding}) takes the form \[z^k\left(z\frac{\d \uu}{\d z}-\zeta \frac{\d \uu}{\d \zeta}\right)=(1-\eta^{k+1})^{-1}\left(\mu \uu-f(z,z^{k+1}\zeta^{k+1},\uu)\right),\] where $\uu(z,\zeta)=\yy(z,z^{k+1}\zeta^{k+1})=\yy(x,\varepsilon)$. By applying Theorem \ref{Thm Main Result} to $P=z^k$ and $L=z\d_z-\zeta \d_\zeta$ we find $\widehat{\uu}(z,\zeta)=\widehat{\yy}(x,\varepsilon)$ is $z^k$-$1$--Gevrey, since $L(P)=kz^k$. Thus, $\widehat{\yy}(x,\eta^{k+1})$ is $\left(\frac{1}{k},\frac{1}{k}\right)$--Gevrey in $(x,\eta)$, and therefore $\widehat{\yy}(x,\varepsilon)$ is $\left(\frac{1}{k},\frac{k+1}{k}\right)$--Gevrey in $(x,\varepsilon)$ as claimed.
\end{eje}

\section*{Acknowledgements} The authors wish to thank the organizers of FASnet20 for their efforts to host this online workshop and for giving us the opportunity to present there the results discussed in this article. We also thank the referee for the careful reading and the many suggestions that improved the paper.

The first author is supported by Ministerio de Econom\'{i}a y Competitividad from Spain, under the Project ``M\'{e}todos asint\'{o}ticos, algebraicos y geom\'{e}tricos en foliaciones singulares y sistemas din\'{a}micos" (Ref.: PID2019-105621GB-I00) and Univ. Sergio Arboleda project IN.BG.086.20.002.

\bibliographystyle{plaindin_esp}

\end{document}